\documentclass[11pt]{amsart}
\usepackage{amsmath, amssymb, mathptmx, mathtools, a4wide}
\usepackage{paralist}
\usepackage[utf8]{inputenc}
\usepackage[T1]{fontenc}

\numberwithin{equation}{section}

\theoremstyle{plain}
\newtheorem{thm}{Theorem}

\newtheorem*{prop*}{Proposition}
\newtheorem{lem}{Lemma}[section]
\newtheorem{prop}[lem]{Proposition}
\newtheorem{cor}[lem]{Corollary}
\newcommand{\thmref}[1]{Theorem~\ref{#1}}
\newcommand{\lemref}[1]{lemma~\ref{#1}}
\newcommand{\propref}[1]{proposition~\ref{#1}}
\newcommand{\corref}[1]{corollary~\ref{#1}}

\theoremstyle{definition}
\newtheorem{rmk}[lem]{Remark}
\newtheorem{defi}[thm]{Definition}

\newcommand{\rmkref}[1]{remark~\ref{#1}}

\def\lf{\left\lfloor}   
\def\rf{\right\rfloor}
\newcommand{\tsum}{\textstyle\sum\limits}
\newcommand{\tprod}{\textstyle\prod\limits}
\newcommand{\D}{\sqrt{|D|}}

\newcommand{\mbf}{\mathbf}
\newcommand{\x}{\textbf}
\newcommand{\mf}{\mathbf}
\newcommand{\q}{\quad}
\newcommand{\qq}{\qquad}

\newcommand{\mc}{\mathcal}

\newcommand{\mrm}{\mathrm}

\newcommand{\sptwo}{\mrm{Sp}_2(\mf Z)}

\newcommand{\gnok}{\mc M_k(\mc{O}_K)}

\newcommand{\plamb}{\mrm{\Lambda}_+(\mc O_K)}
\newcommand{\bpm}{\begin{psmallmatrix}}
\newcommand{\bpml}{\begin{pmatrix}}
\newcommand{\epm}{\end{psmallmatrix}}
\newcommand{\epml}{\end{pmatrix}}
\newenvironment{psm}
  {\left(\begin{smallmatrix}}
  {\end{smallmatrix}\right)}

\begin{document}

\title[Determination of Hermitian cusp forms]{Distinguishing Hermitian cusp forms of degree $2$ by a certain subset of all Fourier coefficients}

\author{Pramath Anamby}
\address{Department of Mathematics\\
Indian Institute of Science\\
Bangalore -- 560012, India.}
\email{pramatha@iisc.ac.in, pramath.anamby@gmail.com}

\author{Soumya Das}
\address{Department of Mathematics\\
Indian Institute of Science\\
Bangalore -- 560012, India.}
\email{soumya@iisc.ac.in, soumya.u2k@gmail.com}


\begin{abstract}
We prove that Hermitian cusp forms of weight $k$ for the Hermitian modular group of degree $2$ are determined by their Fourier coefficients indexed by matrices whose determinants are essentially square-free. Moreover, we give a quantitative version of the above result. This is a consequence of the corresponding results for integral weight elliptic cusp forms, which are also treated in this paper.
\end{abstract}
\maketitle

\section{Introduction}

Recognition results for modular forms has been a very useful theme in the theory. We know that the Sturm's bound, which applies quite generally to a wide class of modular forms, says that two modular forms are equal if (in a suitable sense) their `first' few Fourier coefficients agree. Moreover, the classical multiplicity-one result for elliptic newforms of integral weight says that if two such forms $f_1,f_2$ have the same eigenvalues of the $p$-th Hecke operator $T_p$ for almost all primes $p$, then $f_1=f_2$. Even stronger versions are known, e.g., a result of D. Ramakrishnan \cite{dram} says that primes of Dirichlet density more than $7/8$ suffices.

However, when one moves to higher dimensions, say, to the spaces of Siegel modular forms of degree $2$ onwards, the situation is drastically different. Such a form which is an eigenfunction of the Hecke algebra does not necessarily have multiplicative Fourier coefficients, and multiplicity-one for eigenvalues (in a suitable sense, for $\sptwo$) is not known yet. However the Fourier coefficients, which are indexed by half-integral symmetric positive definite matrices, do determine a modular form. Thus one can still ask the stronger question whether a certain subset, especially one which consists of an arithmetically interesting set of Fourier coefficients, (say e.g., the primitive Fourier coefficients, i.e., those which are indexed by primitive matrices) already determines the Siegel cusp form. These may be considered as a substitute for a ``weak multiplicity-one", as Scharlau-Walling \cite{sch-wall} puts it, in the context of Fourier coefficients.

This line of investigation has attracted the attention of many mathematicians. As a first result in this direction, it was shown by D. Zagier \cite{zagier} that, the Siegel cusp forms of degree $2$ are determined by primitive Fourier coefficients. This has been generalized to Siegel and Hermitian cusp forms with levels and of higher degrees by S. Yamana \cite{yamana}. Similar results along this line, essentially distinguishing Siegel Hecke eigenforms of degree $2$ by the so-called `radial' Fourier coefficients (i.e., by certain subset of matrices of the form $mT$ with $T$ half-integral, $m \geq 1$), has been obtained in Breulmann-Kohnen\cite{br-ko}, Scharlau-Walling \cite{sch-wall}, Katsurada \cite{ibu-katsu}. A result of B. Heim \cite{heim} improves upon some of these results using differential operators on Siegel modular forms of degree $2$. More recently in \cite{saha}, \cite{sa-schm} A. Saha, R. Schmidt has proved that the Siegel cusp forms of degree $2$ are determined (in a quantitative way) by their fundamental (in fact by odd and square-free) Fourier coefficients. 

In this paper we take up the question of determining when two Hermitian cusp forms of degree $2$ on the full Hermitian modular group, which are not necessarily eigenforms, coincide when a certain subset of their Fourier coefficients are the same. This certain set is given explicitly in the theorem stated below, e.g., for $K=\mathbf{Q}(i)$, it consists of all square-free Fourier coefficients up to a divisor of $4$. Let $D_K<0$ be a fundamental discriminant such that $K=\mbf{Q}(\sqrt{D_K})$ has \textit{class number} $\mathit{1}$ (see \rmkref{classno1} for comments on this condition), and $\mc O_K$ be its ring of integers. Recall that in this case $D_K$ belongs to the following set $\{-4, -8, -3, -7, -11, -19, -43, -67, -163\}$.

Let $S_k(\mc O_K)$ denote the space of Hermitian cusp forms of degree $2$ and weight $k$ on the Hermitian modular group $\Gamma_2(\mc O_K)$. Each such cusp form $F$  has a Fourier expansion of the form (see sect.~\ref{defs} for the formal definitions)
\begin{align} \label{2.1}
F(Z) = \underset{T \in \plamb}\sum a(F,T) e \left( \mrm{tr}\text{ } TZ \right), \q \qquad \qquad (e(z) := e^{2 \pi i z} \text{ for } z \in \mf{C}),
\end{align}
where $\plamb := \{T \in M(2, \mf{C}) \mid T=\bar{T}' > 0, \ t_{\mu,\mu} \in \mf{Z}, \ t_{\mu,\nu} \in \tfrac{i}{\sqrt{|D_K|}} \mc O_K  \}$ is the lattice dual to the lattice consisting of $\mc O_K$-integral $2\times 2$ Hermitian matrices with respect to the trace form $\mrm{tr}$. Let us note here that (see sect.~\ref{defs}) $|a(F,T)|$ is invariant under the action $T \mapsto \overline{U}'TU$ ($U \in \mrm{GL}_2(\mc O_K)$), and that $|D_K|\det(T)$ is a positive integer. Further, let $p_{K}$ be the \textit{prime} such that $|D_K|=p_{K}^r$, for some $r\ge 1$. i.e., $p_K=|D_K|$ when $D_K$ is odd and $p_K=2$, when $D_K$ is even. We can now state the main results of this paper.
\begin{thm}\label{th:01}
Let $F\in S_k(\mc O_K)$ be non-zero. Then 
\begin{itemize}
\item [(a)] $a(F,T)\neq 0$ for infinitely many matrices $T$ such that $|D_K|\det(T)$ is of the form $p_K^{\alpha} n$, where $n$ is square-free with $(n,p_K)=1$ and $0\le\alpha\le 2$ if $D_K\neq -8$ and $0\leq \alpha \leq 3$ if $D_K=-8$.
\item[(b)] For any $\varepsilon>0$, \[\#\{0<n<X, \, n \text{ square-free}, (n,p_K)=1, a(F,T)\neq 0, p_K^{\alpha}n= |D_K|\det(T)\}\gg_{F,\varepsilon} X^{1-\varepsilon}.\]
\end{itemize} 
\end{thm}

We say a few words about the proof of the theorem. We assume that $F\neq 0$ and via the Fourier-Jacobi expansion of $F$, reduce the question to Hermitian Jacobi forms of prime index in section~\ref{sec:redhjf}, thanks to a theorem of H. Iwaniec. The standard avenue now would be to pass on to the integral weight forms by using the injectivity of the so-called Eichler-Zagier map (which is essentially the average of all theta components of a Jacobi form). However we stress here that the possibility of this passage to the integral weight forms turns out to be rather non-trivial in our case. 

\textsl{The main point is that even in the case of prime indices, the Eichler-Zagier map (see \eqref{eq:h} for the definition) may not be injective}; unlike the scenario for the classical Jacobi forms. The only result known in this regard is from \cite{haverkamp} that such a map is injective on a certain subspace $\mc J^{spez}_{k,p}(\mc O_K)$ ($p$ prime, see section \ref{subsec:operators}). Moreover \lemref{lem:vpinjective}, \propref{spez} in section~\ref{inj} show that $\mc J^{spez}_{k,p}(\mc O_K)$ may be a proper subspace of $\mc J_{k,p}(\mc O_K)$ and the Eichler-Zagier map may fail to be injective in the complementary space (see remark~\ref{e-z}). 

\textsl{The heart of this paper is devoted to overcome such an obstacle, this is at the same time the second main topic of the paper, treated in detail in section~\ref{hjf}}. Given that our aim is to reduce the question to $S_{k}(N,\chi)$ (the space of cusp forms of weight $k$ on $\Gamma_0(N)$ with character $\chi$) which are pleasant to work with, we consider a `collection' of Eichler-Zagier maps $\iota_\xi$ indexed by suitable characters $\xi$ of the group of units of the ring $\mc O_K/i\sqrt{|D_K|}p \mc O_K$, see section~\ref{sec:indexold} for more details. Each $\iota_\xi$ do map $\mc J^{cusp}_{k,p}(\mc O_K)$ to $S_k(N,\chi)$ for certain $N$ and $\chi$ (see section \ref{sec:introhjf}). Working with this collection of maps, we show that

$(i)$ if the index $p$ of the Hermitian Jacobi form $\phi_p$ at hand is inert in $\mc O_K$, then this `collection' $\{\iota_\xi\}_\xi$ defines an injective map, and 

$(ii)$ if $p$ splits, then either this `collection' is injective or that $\iota$ itself is injective. For this, we have to develop a part of the theory of index-old Hermitian Jacobi forms of index $p$ \`{a} la Skoruppa-Zagier in \cite{skoruppa-zagier}. See section~\ref{sec:indexold}.

Finally $(i)$ and $(ii)$ allow us to reduce the problem to the following theorem on $S_k(N,\chi)$ (the space of cusp forms on $\Gamma_0(N)$ with character $\chi$) for certain $N$ and $\chi$. Results somewhat similar to this have been obtained by Yamana \cite{yamana}, but his results does not imply ours. Thus as far as we know, the following result is not available in the literature. We assume $\chi(-1)=(-1)^k$, so that $S_k(N,\chi)\neq \{0\}$.
\begin{thm}\label{th:02}
Let $\chi$ be a Dirichlet character of conductor $m_\chi$ and $N$ be a positive integer such that $m_\chi|N$ and $N/m_\chi$ is square-free.
\begin{itemize}
\item[(a)] If $f\in S_k(N,\chi)$ and $a(f,n)=0$ for all but finitely many square-free integers $n$. Then $f=0$.
\item[(b)] Let $f\in S_k(N,\chi)$ and $f\neq 0$, then for any $\epsilon>0$ \[\#\{0<n<X, \text{n square-free},a(f,n)\ne 0\}\gg_{f,\epsilon} X^{1-\epsilon}.\]
\end{itemize} 
\end{thm}

Clearly, part~$(a)$ of \thmref{th:02} follows from part ~$(b)$, however we include an independent proof of part~$(a)$ using an argument adapted from the work of Balog-Ono \cite{BalogOno}, which we feel is worth noting and the method could be useful in other circumstances. In a nutshell and loosely speaking, this method allows one to reduce to the case of newforms. In either of the proofs, the condition on the ratio of the level and conductor is necessary, this can be seen by taking the example of a non-zero form $g(\tau) \in S_k(\mrm{SL}_2(\mf Z))$ and consider $g(m^2 \tau)$ for some $m > 1$. The proofs of these results are given in section~\ref{int}. Let us mention here that motivated by \thmref{th:02} and with the same hypotheses, very recently we could prove that there exists a constant $B$ depending only on $k,N$ such that if $a_f(n)=0$ for all square-free $n \leq B$, then $f=0$.

For the proof of part~$(b)$, we essentially consider the cusp form obtained from a given form by sieving out squares and then apply the Rankin-Selberg method to get asymptotics of the second moment of its Fourier coefficients; the details are rather technical, see sections~\ref{2mom} and~\ref{partb}. Along the way, we present some nice calculations on the Petersson norms of $U_{r^2}f$, which arise as a part of the main term in the asymptotic alluded to above, with $f$ as in the theorem, and which extends the results of \cite{brown}.

Finally we remark that with some modifications, one expects to extend our results to the corresponding spaces of Eisenstein series as well; it could be interesting to work this out.

\noindent\textbf{Acknowledgements.} It is a pleasure to thank Prof. S. B\"ocherer for his comments and encouragement about the topic of the paper. The first author is a DST- INSPIRE Fellow at IISc, Bangalore and acknowledges the financial support from DST (India). The second author acknowledges financial support in parts from the UGC Centre for Advanced Studies, DST (India) and IISc, Bangalore during the completion of this work.
\section{Notation and terminology}
We mostly follow standard notation throughout the paper:  $M(n, R)$ denotes, as usual, the space of $n \times n$ matrices over a commutative ring $R$; for $A \in M(n, \mf{C})$, $A^{*} := \bar{A}'$, with $A'$ denoting the transpose of $A$; $A$ is Hermitian if $A=A^*$ and is positive definite (resp. semi--definite) if $\xi^* A \xi > 0$ (resp. $\geq 0$) for all $\xi \in \mf{C}^n \backslash \{0\}$. 

\subsection{Hermitian modular forms} \label{defs}
We define the unitary group of degree $2$ as
\[   U(2,2) := \{ M \in \mrm{GL}(4, \mf{C}) \mid \bar{M}' J M = J  \},   \]
where $J = \left( \begin{smallmatrix} 0 & -I_2 \\ I_2 & 0  \end{smallmatrix}  \right)$. We recall the Hermitian upper half-space of degree 2 on which most of the holomorphic functions in this paper live:
\[  \mc{H}_2 := \{ Z  \in M(2,\mf{C}) \mid (Z-Z^*)/{2i} > 0 \}.  \]

Let $D_K$ be a fundamental discriminant and $K$ denote an imaginary quadratic field of discriminant $D_K$, i.e., $K = \mf{Q}(\sqrt{D_K})$. The class number of $K$ is assumed to be $1$. Denote the ring of integers of $K$ by $\mc \mc \mc O_K$ and the order of the unit group $\mc \mc \mc O_K^{\times}$ of $\mc \mc \mc O_K$ by $w(D_K)$. The inverse different of $K$ is denoted by \[ \mc \mc \mc O_K^{\#} := \tfrac{i}{\sqrt{|D_K|}} \mc \mc \mc O_K. \]
We denote by $ \Gamma_2( \mc \mc \mc O_K)$ the Hermitian modular group of degree $2$ defined by
\[  \Gamma_2(\mc \mc \mc O_K) := U(2,2) \cap M(4, \mc \mc \mc O_K). \]

Given an integer $k$, the vector space of Hermitian modular forms of degree 2 and weight $k$ consists of all holomorphic functions $f \colon \mc{H}_2 \rightarrow \mf{C}$ satisfying 
\[ f(Z) = \det (CZ+D)^{-k} f(M \langle Z \rangle) \q  \text{for all  } Z \in \mc{H}_2,\text{ } M=\begin{psm}A&B\\C&D\end{psm} \in \Gamma_2(\mc \mc \mc O_K). \]
where $M \langle Z \rangle:=(AZ+B)(CZ+D)^{-1}$. The vector space of Hermitian modular forms of degree $2$ (with respect to $K$) and weight $k$ is denoted by $\gnok$. Further, those forms in $\gnok$ which have Fourier expansion as in \eqref{2.1} are cusp forms and the subspace of all cusp forms is denoted by $S_k(\mc O_K)$.

Moreover, following Yamana \cite{yamana} let us define the content $c(T)$ of a matrix $T \in \plamb$ by
\[ c(T) := \max \{ a \in \mf N \mid a^{-1} T \in \plamb \}. \]
$T\in \plamb$ is called primitive, if $c(T)=1$.

Expanding an $F \in S_k(\mc O_K)$ along the Klingen parabolic subgroup, we can write its Fourier-Jacobi expansion as
\begin{equation}\label{eq:1}
F(Z)=\underset{m\ge 1}{\sum}\phi_{m}(\tau,z_{1},z_{2})e(m\tau'),
\end{equation}
where $Z=\begin{psmallmatrix}
\tau & z_{1}\\
z_{2} & \tau'
\end{psmallmatrix}$ and for each $m\ge 1$, the Fourier-Jacobi coefficient $\phi_{m} \in \mc{J}_{k,m}^{cusp}(\mc \mc \mc O_K)$ with
\begin{equation}\label{eq:2}
\phi_{m}(\tau,z_{1},z_{2})=\underset{nm>N(r)}{\underset{n\in \mathbf{Z},r\in \mc \mc \mc O_K^{\#}}{\sum}}a\left( F, \begin{psmallmatrix}
n & r\\
\overline{r} & m
\end{psmallmatrix}\right) e(n\tau + \overline{r}z_{1} + rz_{2}),
\end{equation}
where $N(\cdot)$ is the norm function of $K$ and $\mc{J}_{k,m}^{cusp}(\mc \mc \mc O_K)$ is the space of Hermitian Jacobi cusp forms for the group $\Gamma^J(\mc \mc \mc O_K)$ (see next section for details).
\subsection{Hermitian Jacobi forms}\label{sec:introhjf}
\subsubsection*{The Hermitian-Jacobi group:} Let $S^1$ denote the unit circle. Then the set $\mbf{C}^2\times S^1$ is a group with the following twisted multiplication law, which we would use freely throughout the paper. \[ [(\lambda_1, \mu_1),\xi_1]\cdot[(\lambda_2, \mu_2),\xi_2]:=[(\lambda_1+\lambda_2, \mu_1+\mu_2),\xi_1\xi_2\ e(2Re(\overline{\lambda_1}\mu_2))].\]
The group $U(1,1)=\{\varepsilon M\mid \varepsilon\in S^1,M\in\mrm{SL}_2(\mbf R)\}$ acts on $\mbf{C}^2\times S^1$ as\[[(\lambda, \mu),\xi](\varepsilon M):=[(\overline{\varepsilon}\lambda, \overline{\varepsilon}\mu)M, \xi e(abN(\lambda)+cdN(\mu)+2bcRe(\overline{\lambda}\mu))].\]
Let $\mc{G}^J$ denote the semi-direct product $U(1,1)\ltimes (\mbf{C}^2\times S^1)$. The multiplication in $\mc{G}^J$ is given by \[[\varepsilon_1 M_1,X_1][\varepsilon_2 M_2,X_2]=[\varepsilon_1\varepsilon_2 M_1M_2,(X_1(\varepsilon_2M_2))\cdot X_2].\]
$\mc{G}^J$ acts from left on $\mc{H}\times\mbf{C}^2$ and from right on functions $\phi:\mc{H}\times\mbf{C}^2\longrightarrow \mbf{C}$. These actions are given by individual actions of $U(1,1)$ and $\mbf{C}^2\times S^1$ as below.
\begin{equation}\label{eq:gpact1}
\begin{aligned}
\varepsilon M (\tau,z_1,z_2)&:=(M\tau,\tfrac{\varepsilon z_1}{c\tau+d},\tfrac{\overline{\varepsilon} z_2}{c\tau+d}).\\
[(\lambda\ \mu ),\xi](\tau,z_1,z_2)&:=(\tau,\ z_1+\lambda \tau+\mu,\ z_2+\overline{\lambda} \tau+\overline{\mu}).
\end{aligned}
\end{equation}
\begin{equation}\label{eq:gpact2}
\begin{aligned}
(\phi|_{k,m} \varepsilon M)(\tau,z_{1},z_{2}) &:= \varepsilon^{-k}(c\tau + d)^{-k} e^{\frac{-2\pi i m c z_{1}z_{2}}{c\tau + d}} \phi \left(M\tau, \tfrac{\varepsilon z_{1}}{c\tau + d},\tfrac{\bar{\varepsilon} z_{2}}{c\tau + d}\right).\\
(\phi|_{m}[(\lambda\ \mu),\xi])(\tau,z_1,z_2) &:=\xi^m e^{2 \pi i m (N(\lambda)\tau + \overline{\lambda}z_{1}+\lambda z_{2})}
\phi (\tau,z_{1} + \lambda \tau + \mu, z_{2} + \overline{\lambda} \tau + \overline{\mu}).
\end{aligned}
\end{equation}
Here $ M = \left( \begin{smallmatrix}a & b \\ c & d  \end{smallmatrix}\right) \mbox{ in } \mrm{SL}_2(\mathbf{R})$, $M\tau=\tfrac{a\tau+b}{c\tau+d}$ and $k,m\in\mbf{Z}$.

The Hermitian-Jacobi group $\Gamma^J(\mc O_K)$ is defined as $\Gamma^J(\mc{O}_K):= \Gamma_1(\mc{O}_K) \ltimes \mc{O}_K^{2}$, where \[ \Gamma_1(\mc{O}_K) := \{ \varepsilon \mrm{SL}_2(\mathbf{Z}) \mid \epsilon \in  \mc{O}_K^\times \}\subset U(1,1)\] and $\mc{O}_K^2=\{(\lambda,\mu)\mid \lambda,\mu\in \mc{O}_K\}$ is the subgroup of $\mbf{C}^2\times S^1$ with component wise addition (here $(\lambda,\mu)$ is identified with $[(\lambda,\mu),1]$).

For positive integers $k$ and $m$, the space of Hermitian Jacobi forms of weight $k$ and index $m$ for the group $\Gamma^{J}(\mc \mc \mc O_K )$ consists of holomorphic functions $\phi$ on $ \mathcal{H} \times \mathbf{C}^{2}$ such that (see \cite{haverkamp})
\begin{enumerate}
\item $\phi|_{k,m}\gamma=\phi, \qq \text{ for all } \gamma\in  \Gamma^J(\mc O_K)$.

\item $\phi$ has a Fourier expansion of the form
\begin{equation*}
\phi(\tau,z_{1},z_{2}) = \sum_{n = 0}^{\infty} \underset{\underset{nm \geq N(r)}{r \in \mc \mc \mc O_K^{\#}}}\sum c_{\phi}(n , r) e\left( n \tau + r z_{1} + \overline{r} z_{2} \right).
\end{equation*}
\end{enumerate}
The complex vector space of Hermitian Jacobi forms of weight $k$ and index $m$ is denoted by $\mc{J}_{k,m}(\mc \mc \mc O_K )$. Moreover, if $ c_{\phi}(n , r) = 0$ for $nm = N(r)$, then $\phi$ is called a Hermitian Jacobi \textit{cusp} form. The space of Hermitian Jacobi cusp forms of weight $k$ and index $m$ is denoted by $\mc{J}_{k,m}^{cusp}(\mc{O}\sb{K} )$. 

For the rest of the paper, for the sake of simplicity we just write $\mc{O}$ instead of $\mc{O}_K$, $D$ instead of $D_K$ and $\mc{J}_{k,m}$ instead of $\mc{J}_{k,m}(\mc \mc \mc O_K )$. Since $\mc \mc \mc O_K^{\#} = \tfrac{i}{\D} \mc \mc \mc O_K$, if $\phi\in \mc{J}_{k,m}$ we can rewrite the Fourier expansion of $\phi$ equivalently as
\begin{equation}\label{fourier}
\phi(\tau,z_{1},z_{2}) = \sum_{n = 0}^{\infty} \underset{|D| nm \geq N(r)}{\sum_{r\in \mc{O}}} c_{\phi}(n , r) e\big( n \tau + \tfrac{i r}{\D} z_{1} + \tfrac{\overline{i r}}{\D} z_{2} \big).
\end{equation}

\subsubsection*{Theta decomposition} As in the case of classical Jacobi forms, Hermitian Jacobi forms admit a theta decomposition. Let $\phi \in \mc{J}_{k,m}$ has the Fourier expansion as in (\ref{fourier}). Then we have 
\begin{equation}\label{eq:thetadec}
\phi(\tau,z_{1},z_{2}) = \underset{s \in \mc{O}/i\D m \mc{O}}\sum h_{s}(\tau) \cdot \theta_{m,s}(\tau,z_{1},z_{2}),
\end{equation}
where, for $s$ as above
\begin{align}
\label{def:theta}\theta_{m,s}(\tau,z_{1},z_{2})&:= \underset{r \equiv s\pmod{i\D m}}\sum e \big(\tfrac{N(r)}{|D|m}\tau + \tfrac{ir}{\D}z_{1}+\tfrac{\overline{ir}}{\D}z_{2} \big).\\
\label{def:hs}h_{s}(\tau)&:=\underset{N(s)+n\in |D|m\mathbf{Z}}{\sum_{n>0}}c\big(\tfrac{n+N(s)}{|D|m},s\big)e(n\tau /|D|m).
\end{align}
The theta components $h_{s}$ of $ \phi \in \mc{J}_{k,m}$ (see \cite{haverkamp, sasaki}) have the following transformation properties under $\mrm{SL}_2(\mbf{Z})$ and $\mc{O}^{\times}$:
\begin{align}
h_s(\tau+1)&=e\Big(-\tfrac{N(s)}{|D|m}\Big)h_s.\\
\label{eq:units}\epsilon^kh_{\epsilon s}(\tau)&=h_s(\tau) \q\text{ where } \epsilon\in \mc{O}^\times.\\
\label{hinversion}h_s(-\tau^{-1})&=\tfrac{i}{\D m}\tau^{k-1}\sum_{r\in \mc{O}/i\D m\mc{O}}e\left(\tfrac{2Re(s\overline{r})}{|D| m}\right)h_r(\tau).
\end{align}

Let $\chi_{D}:=\left(\frac{D}{\cdot}\right)$, the unique real primitive Dirichlet character mod $|D|$. Then for any $M=\begin{psm}a&b\\ c&d\end{psm}\in \Gamma_0(m|D|)$ and $J=\begin{psm}0&1\\ -1&0\end{psm}$, we have
\begin{equation}\label{eq:transtheta}
\theta_{m,s}\mid _{1,s}MJ=\tfrac{i\chi_D(d)}{m\D}\sum_{s'\in \mc{O}/i\D m\mc{O}}e(a(bN(s)+2Re(\overline{s}s'))/|D|m)\theta_{m,s'}.
\end{equation}

\subsubsection*{An exponential sum} For $K,D$ as above, we would encounter the following exponential sum. Its evaluation is standard, so we just state it.
\begin{align} \label{expo}
\underset{r\in \mc{O}/ s \mc{O}}{\sum}e \big(2Re\big(\tfrac{ irx}{\D s} )\big)= \begin{cases}  N(s), & \text{if  } x \in s\mc O; \\ 0, & \text{otherwise} . \end{cases}  
\end{align}
\subsubsection*{Eichler-Zagier maps} Using the theta decomposition for $\phi\in\mc{J}_{k,m}$ as in (\ref{eq:thetadec}) define the Eichler-Zagier map $\iota:\mc{J}_{k,m}\longrightarrow S_{k-1}(|D|m,\chi_{D})$ by $\iota(\phi)=h$, where (see \cite{ez} for the classical case and \cite{haverkamp} for more details)
\begin{equation}\label{eq:h}
h(\tau):=\sum_{s\in \mc{O}/i\D m\mc{O}}h_s(|D|m\tau).
\end{equation}

Let $\mathcal{D}=\D i\mc{O}$ denote the different of $\mathbf{Q}(\sqrt{D})$ and define the subgroup $G$ of $\left(\mc{O}/i\D m\mc{O}\right)^\times$ by
\begin{equation} \label{Gdef}
G:=\{\mu+m\mathcal{D} \mid N(\mu)\equiv 1 \mod {|D|m}\}.
\end{equation}
Let $\eta :G\longrightarrow \mathbf{C}$ be any character of $G$ such that $\eta(\varepsilon)=\varepsilon^{-k}$ for all $\varepsilon\in \mc{O}^\times$. Let $\tilde{\eta}$ be an extension of $\eta$ to $(\mc{O}/i\D m\mc{O})^\times$. Now define the \textit{twisted} Eichler-Zagier map $\iota_{\tilde{\eta}}:\mc{J}_{k,m}\longrightarrow S_{k-1}(2f|D|m,\chi_{D}\cdot\overline{\tilde{\eta}})$ by $\iota_{\tilde{\eta}}(\phi)=h_{\tilde{\eta}}(\tau)$, where
\begin{equation}\label{eq:heta}
h_{\tilde{\eta}}(\tau):=\sum_{s\in \mc{O}/i\D m\mc{O}}\overline{\tilde{\eta}(s)}h_s(|D|m\tau)
\end{equation}
and $f \in \mf Z \cap i\D m\mc{O}$. We choose $f$ to be the minimal such positive integer, so that $f=|D|m$ when $D$ is odd and $f=\frac{|D|m}{2}$ when $D$ is even. For the convenience of the reader we indicate how one can prove that $h_{\tilde{\eta}}\in S_{k-1}(2f|D|m, \chi_{D}\cdot \overline{\tilde{\eta}})$. 

Namely, for any $M=\begin{psm}a&b\\ 2cf|D|m&d\end{psm}\in \Gamma_0(2f|D|m)$ and $f$ as above, 
\begin{equation}\label{eq:hetatildetr}
h_{\tilde{\eta}}\mid_{k-1}M(\tau)=\sum_{s\in \mc{O}/i\D m\mc{O}}\overline{\tilde{\eta}(s)}\Big(h_s\mid_{k-1}\begin{psm}a&b|D|m\\ 2cf&d\end{psm}\Big) (|D|m\tau).
\end{equation}
Now using the transformation formula (\ref{eq:transtheta}) for $\theta_{m,s}$ , we have 
\begin{align*}
h_s\mid_{k-1}\begin{psm}a&b|D|m\\ 2cf&d\end{psm}&=h_s\mid_{k-1}J\begin{psm}-d&2cf\\ b|D|m&-a\end{psm}J\\
&=\tfrac{\chi_D(d)}{|D|m^2}\sum_{s',s''\in \mc{O}/i\D m\mc{O}}e((2cdfN(s')-2Re(\overline{s}s')+2dRe(\overline{s''}s'))/|D|m)h_{s''}.
\end{align*}
Using this in (\ref{eq:hetatildetr}) and evaluating the exponential sum over $s'$ from \eqref{expo} we infer that $h_{\tilde{\eta}}\in S_{k-1}(2f|D|m, \chi_{D}\cdot \overline{\tilde{\eta}})$.

\subsubsection{Decomposition of $\mc {J}_{k,m}$.}\label{sec:decomp}
For $\mu\in \mc O$ with $N(\mu)\equiv 1\pmod {m|D|}$ define
\begin{equation*}
W_\mu(\phi):=\underset{s \in \mc{O}/i\D m \mc{O}}\sum h_{\mu s}(\tau) \cdot \theta_{m,s}(\tau,z_{1},z_{2}),
\end{equation*}
where $\phi\in\mc{J}_{k,m}$ and has theta decomposition as in (\ref{eq:thetadec}). Then $W_\mu$ is an automorphism of $\mc{J}_{k,m}$.

Let $G$ be the group defined as above. Then the map $G\rightarrow End(\mc{J}_{k,m})$, $\mu\mapsto W_\mu$ is a homomorphism. Now as in the case of classical Jacobi forms we can decompose $\mc{J}_{k,m}$ as
\begin{equation*}
\mc{J}_{k,m}=\underset{\eta}{\oplus}\mc{J}_{k,m}^\eta,
\end{equation*}
where $\eta$ is a character of $G$ as above and
\begin{equation*}
\mc{J}_{k,m}^\eta:=\{\phi\in\mc{J}_{k,m}|W_\mu(\phi)=\eta(\mu)\phi \text{ for all } \mu\in G\}.
\end{equation*}
Now let $\eta_0$ is the trivial character of $G$. For $\eta\neq \eta_0$, let $\phi\in \mc{J}_{k,m}^{\eta}$. Then we have $W_\mu(\phi)=\eta(\mu)\phi$ for all $\mu\in G$. That is $h_{\mu s}=\eta(\mu)h_s$ for all $\mu \in G$. Note that $\mu$ is an unit in $\mc{O}/i\D m\mc{O}  $. Thus if $h=\iota(\phi)$ is defined as in (\ref{eq:h}), then 
\begin{align*}
h(\tau)&=\sum_{s\in \mc{O}/i\D m\mc{O}}h_s(|D|m\tau)\underset{s\mapsto\mu s}{=}\sum_{s\in \mc{O}/i\D m\mc{O}}h_{\mu s}(|D|m\tau)\\
&=\eta(\mu)\sum_{s\in \mc{O}/i\D m\mc{O}}h_s(|D|m\tau)=\eta(\mu)h(\tau).
\end{align*}
Since $\eta\neq\eta_0$, we have $h=0$. This implies that $\underset{\eta\neq\eta_0}{\oplus}\mc{J}_{k,m}^\eta\subset ker(\iota)$. Similarly for any non trivial character $\eta$ of $G$ it follows that $\underset{\eta '\neq\eta}{\oplus}\mc{J}_{k,m}^{\eta '}\subset ker(\iota_{\tilde{\eta}})$.

\subsection{Elliptic modular forms}
For a positive integers $k$, $N$ and a Dirichlet character $\chi \mod N$, let $S_{k}(N,\chi)$ denote the space of cusp forms of weight $k$ and character $\chi$ for the group $\Gamma_{0}(N)$.

For $f\in S_{k}(N,\chi)$ we write its Fourier expansion as
\begin{equation*}
f(\tau)=\sum_{n=1}^{\infty}a'(f,n)n^{\frac{k-1}{2}}e(n\tau),
\end{equation*}
so that by Deligne \cite{deligne}, we have the estimate for any $\varepsilon>0$:
\begin{align}\label{eq:deligne}
|a'(f,n)| \ll_{\varepsilon,f} n^{\varepsilon}.
\end{align}
For a positive integer $n$ with $(n,N)=1$, the Hecke operator $T_n$ on $S_{k}(N,\chi)$ is defined by
\begin{equation}\label{hecke}
T_nf=n^{\frac{k}{2}-1}\underset{a>0}{\sum_{ad=n}}\chi(a)\sum_{b=0}^{d-1}f|\begin{psm}
a & b\\
0 & d
\end{psm}.
\end{equation}
For any $n$, the operator $U_n$ is defined as
\begin{equation}\label{heckeUn}
U_nf=n^{\frac{k}{2}-1}\sum_{b=0}^{n-1}f|\begin{psm}
1 & b\\
0 & n
\end{psm}.
\end{equation}

The space $S_k(N,\chi)$ is endowed with Petersson inner product defined by
\begin{equation}
\left\langle f,g\right\rangle_N=\int_{\Gamma_{0}(N)\backslash\mc{H}} f(\tau)\overline{g(\tau)}y^{k-2}dxdy.
\end{equation}

\section{Proof of \thmref{th:01}}\label{sec:redhjf}
\subsection{Reduction to Hermitian Jacobi forms.}\label{sec:thm1}
In order to prove \thmref{th:01}, as is quite natural (see also \cite{saha}) we first reduce the question to the setting of Hermitian Jacobi forms of prime index. This would be possible, as is explained later, if we could show that any matrix in $\Lambda_{+}(\mc{O})$ is equivalent to one with the right lower entry an odd prime. The following lemma allows us to do that. To prove the lemma, we crucially use the following very non-trivial result due to H. Iwaniec \cite{iwa} on primes represented  by a general primitive quadratic polynomial of $2$ variables, stated in a way to suit our need.
\begin{thm} \label{life}
Let $P(x,y)= Ax^2 +Bxy + Cy^2 + Ex + Fy + G  \in \mf Z[x,y]$ be such that $(A,B,C,E,F,G)=1$. If $P$ is irreducible in $\mf Q[x,y]$ and represents arbitrarily large odd integers and depend essentially on two variables, then it represents infinitely many odd primes.
\end{thm}

In the above theorem, $P(x,y)$ is said to depend essentially on two variables if $(\partial P/\partial x)$ and $(\partial P/\partial y)$ are linearly independent.
\begin{lem}[Hermitian forms representing primes] \label{indp}
Let $T\in \Lambda_{+}(\mc{O})$ be a primitive matrix. Then there exist $g \in \mrm{GL}_2(\mc{O})$ such that $g^*Tg = \bpm * & * \\ * & p \epm$ for some odd prime $p$.		
\end{lem}

\begin{proof}
Let us write $g = \bpm \alpha & \beta \\ \gamma & \delta \epm$ and $T = \bpm n & r \\ \overline{r} & m \epm \in \Lambda_{+}(\mc{O})$. Then one computes that
\[g^*Tg =\bpml * & * \\ * &  N(\beta)n + \delta r \overline{\beta}  +   \beta\overline{r} \overline{\delta}   +N(\delta) m   \epml .\]
At this point we would like to invoke \thmref{life}, choosing $g$ appropriately according to the following cases.

\noindent $\bullet$ \textit{Either $m$ or $n$ is odd:} If $m$ is odd, set $\delta=1,\gamma=0$. Such a matrix can be easily completed to $\mrm{GL}_2(\mc{O})$ for any value of $\beta$.

When $D\equiv 0 \pmod 4$, $r$ is of the form $r=\tfrac{i}{\D}(r_1+\frac{i}{2}\D r_2)$ and set $\beta=x+\frac{i}{2}\D y$. We put
\[ P(x,y) = n(x^2+\tfrac{|D|}{4}y^2) - r_2x+r_1y+m. \]
When $D\equiv 1 \pmod 4$,  r is of the form $r=\tfrac{i}{\D}(\frac{r_1}{2}+\tfrac{i\D}{2}r_2)$ and set $\beta=x+i\D y$. In this case we put
\[ P(x,y) = n(x^2+|D|y^2) - r_2x+r_1y+m. \]
Noting that $T$ is primitive, it is easily seen that in both cases $P$ satisfies the first hypothesis of \thmref{life}. Hence it is enough to prove that $P(x,y)$ is irreducible in $\mathbf{Q}$.

If at all there is a non-trivial factorization over $\mf Q$, it has to be into two linear factors, say
\[ P(x,y) = (a_1 x +b_1 y + c_1) ( a_2 x +b_2 y + c_2 ).\]
A short calculation shows that $a_1 / b_1 = - a_2/ b_2 = \lambda \q \text{(say)};$
Now comparing the coefficients of $x^2$ and $y^2$ we get that $\lambda^2=
-\frac{4}{|D|},  \text{ when } D\equiv 0\pmod 4 \text{ and }
\lambda^2=-\frac{1}{|D|},  \text{ when } D\equiv 1\pmod 4.$
A contradiction in both the cases. Hence $P$ is irreducible.

Further that $P$ represents arbitrarily large odd values is clear since $m$ is odd and we can vary $x,y$ over large even integers. Essential dependence in two variables is trivial in our case. Thus $P$ represents infinitely many odd primes. The case when $n$ is odd follows by symmetry of the situation (we take $\beta=1$ and proceed similarly).

\noindent $\bullet$ \textit{$m$ and $n$ are even:} In this case we can take $P$ as before and note that one of $r_1$ or $r_2$ must be odd, since $T$ was primitive. Say $r_2$ is odd. Then varying $x$ through odd integers and $y$ through even ones, we see that $P$ represents arbitrarily large odd integers. The other properties of $P$ continue to hold.
\end{proof}

We embark upon the proof of \thmref{th:01} by using the following result due to S. Yamana~\cite{yamana}.
\begin{thm}\label{yamana}
If $F \in S_k({\mc O})$ is non-zero, then there exists a primitive $T \in \Lambda_{+}(\mc{O})$ such that $a(F,T) \neq 0$.
\end{thm}

\subsection{Reduction to elliptic cusp forms and proof of \thmref{th:01}}\label{subsec:thm1}
Let $F\in S_{k}(\mc{O})$ be non-zero and by \thmref{yamana}, choose $T_0 \in \Lambda_{+}(\mc{O})$ primitive such that $a(F,T_0) \neq 0$. From fact that $a(F, g^*Tg)=(\text{det }g)^k a(F,T)$ for all $g \in \mrm{GL}_2(\mc{O})$ and by using \lemref{indp} with $T=T_0$, we can assume that $T_0=\bpm * & * \\ * & p  \epm$ for an odd prime $p$.

Appealing to the Fourier-Jacobi expansion of $F$ as in \eqref{eq:1} and the above conclusion, it follows that there is an odd prime $p$ with $(p,i\D)=1$ such that $\phi_{p}\in \mc{J}_{k,p}$ is non-zero. Recall that the Fourier expansion of $\phi_{p}$ has the shape
\begin{equation}
\phi_p(\tau,z_{1},z_{2}) = \sum_{n = 0}^{\infty} \underset{|D| np \geq N(r)}{\sum_{r\in \mc{O}}} c_F(n , r) e\left( n \tau + \tfrac{i r}{\D} z_{1} + \tfrac{\overline{i r}}{\D} z_{2} \right),
\end{equation}
where $c_F(n,r)=a\left( F, \bpm
n & ir/\D\\
\overline{ir}/\D & p
\epm \right)$.

Now let $h^F$ and $h_{\tilde{\eta}}^F$ be the images of $\phi_{p}$ under $\iota$ and $\iota_{\tilde{\eta}}$ respectively (defined in section \ref{sec:introhjf}). The crucial fact is the following, proved at the end of section~\ref{sec:indexold}.
\begin{prop}\label{prop:finalprop}
Let $p\in \mbf{Z}$ be a prime and $\phi\in \mc{J}_{k,p}$ be non zero. Then $\iota_{\tilde{\eta}}(\phi)\neq 0$ for some $\eta$ or $\iota(\phi)\neq 0$.
\end{prop}

Now suppose that $h^F\neq 0$. Let the Fourier expansion of $h^F$ be given by $h^F(\tau)=\sum_{n>0}A(n)e^{2\pi in\tau}$, where $A(n)$ is given by
\begin{equation*}
A(n)=\underset{N(s)+n\in |D|p\mathbf{Z}}{\sum_{s\in \mc{O}/i\D p\mc{O}}}c_F\left(\tfrac{n+N(s)}{|D|p},s\right).
\end{equation*}
Since $h^F\neq 0$ and $N/m_\chi$ is square-free, using \thmref{th:02}(a), we get infinitely many square-free $n$ such that $A(n)\neq 0$. For each of these $n$, we get an $s$ such that $c_F\left(\frac{n+N(s)}{|D|p},s\right)=a( F, \begin{psmallmatrix}
\frac{n+N(s)}{|D|p} & is/\D\\
\overline{is}/\D & p
\end{psmallmatrix})$ is not equal to zero. Moreover by \thmref{th:02}(b), for any $\varepsilon>0$ we have
\[\#\{0<n<X, \text{ }n \text{ square-free, }A(n)\neq 0\}\gg_{h^F,\varepsilon} X^{1-\varepsilon}.\]
Thus, for any $\varepsilon>0$, $\#\{0<n<X,\, n \text{ square-free},\, a(F,T)\neq 0,\, n=|D|\mathrm{det}(T)\}\gg_{F,\varepsilon} X^{1-\varepsilon}.$

Now suppose $h^F=0$, then by proposition \ref{prop:finalprop}, there exists a character $\eta$ of $G$ such that $h_{\tilde{\eta}}^F\neq 0$. We need another proposition, whose proof is deferred to end of section~\ref{subsec:lemmaG}.
\begin{prop}\label{prop:hetanonzero}
Let $\eta$ be a character of $G$. Suppose $\iota_{\tilde{\eta}}(\phi)\neq 0$ for some extension $\tilde{\eta}$ of $\eta$, then there exists an extension $\tilde{\eta}_0$ of $\eta$ such that restriction of $\tilde{\eta}_0$ to $\mbf {Z}$ has conductor divisible by $p$ and $\iota_{\tilde{\eta}_0}(\phi)\neq 0$.
\end{prop} 

Note that there is a choice in extending $\eta$ to $\tilde{\eta}$. But different $\iota_{\tilde{\eta}}(\phi)$ obtained in this way are either all vanish or none of them can vanish (see \lemref{lem:twists}). This allows us to assume that $h^F_{\tilde{\eta}}$ satisfies the conditions in \propref{prop:hetanonzero}.

We can write $h_{\tilde{\eta}}^F(\tau)=\sum_{n>0}B(n)q^n$, where $B(n)$ is given by
\[B(n)=\underset{ s\in \mc{O}/i\D p\mc{O} , \, N(s)+n\in |D|p\mathbf{Z} }  \sum  \overline{\tilde{\eta}(s)}c_F\left(\tfrac{n+N(s)}{|D|p},s\right).\]

\textsl{Case 1: } When $D$ is odd, $2f|D|p \big / m_{\chi_{D}\cdot \overline{\tilde{\eta}}}$ is of the form $|D|^\alpha 2p$, where $1\le \alpha\leq 2$.

\noindent$\bullet$ If $\alpha=1$, then we can apply \thmref{th:02} to $h^F_{\tilde{\eta}}$ and we get the result.

\noindent$\bullet$ If $\alpha=2$, then we apply proposition \ref{prop:genth2} (please see the end of section~\ref{int}) to $h^F_{\tilde{\eta}}$ with $p_1=2$, $p_2=|D|$, $p_3=p$ and $\alpha_1=1$, $\alpha_2=\beta=2$ and we get the result.

\textsl{Case 2:} When $D$ is even, $2f|D|p \big / m_{\chi_{D}\cdot \overline{\tilde{\eta}}}$ is of the form $|D|p$, since $\chi_{D}\cdot \overline{\tilde{\eta}}$ is a primitive character $\pmod {|D|p}$. We use proposition \ref{prop:genth2} for $h^F_{\tilde{\eta}}$, with $p_2=2$, $p_3=p$, $\alpha_1=0$ and $\alpha_2=4$, $\beta=2$, when $D=-4$ and $\alpha_2=6$, $\beta=3$, when $D=-8$ to get the result.\qed

\section{Interlude on Hermitian Jacobi Forms} \label{hjf}
\subsection{Some operators on $\mc{J}_{k,m}$}\label{subsec:operators}
In order to proceed further we need a few operators on $\mc{J}_{k,m}$. Let $\rho\in \mc{O}$ ($\rho\neq 0$), define the Hecke-type operator $U_\rho:\mc{J}_{k,m}\longrightarrow \mc{J}_{k,mN(\rho)}$ by (\cite[p.51]{haverkamp})
\begin{equation}
\phi|U_\rho(\tau,z_1,z_2)=\phi(\tau,\rho z_1,\overline{\rho}z_2).
\end{equation}
If $\phi$ has a Fourier expansion as in (\ref{fourier}), then the Fourier expansion of $\phi|U_\rho$ is given by
\begin{equation}\label{eq:Up}
\phi|U_\rho(\tau,z_1,z_2)=\sum_{n = 0}^{\infty} \underset{|D|N(\rho)nm \geq N(r)}{\sum_{r\in \rho \mc{O}}} c_{\phi}(n , r/\rho) e\left( n \tau + \tfrac{ir}{\D} z_{1} + \tfrac{\overline{ir}}{\D} z_{2} \right).
\end{equation}

Now for $\rho\in \mc{O}$ with $\rho|m$ and $N(\rho)|m$ we define a new operator $u_\rho$ on $\mc{J}_{k,m}$ as given below:
\begin{equation}
\phi|u_{\rho}(\tau,z_1,z_2):=N(\rho)^{-1}\sum_{x\in \mc{O}^2/\rho \mc{O}^2}\left(\phi\mid_{k,m}\left[\tfrac{x}{\rho}\right]\right)(\tau,z_1/\rho,z_2/\overline{\rho}).
\end{equation}
\begin{lem}
Let $u_\rho$ be defined as above. Then $u_\rho$ is an operator from $\mc{J}_{k,m}$ to $\mc{J}_{k,m/N(\rho)}$.
\end{lem}
\begin{proof}
Let $\varepsilon M\in\Gamma_1(\mc{O})$ and $[\lambda,\mu]\in \mc{O}^2$. Then the requisite transformation properties of $\phi|u_\rho$ easily follow since if $\{x=(x_1,x_2)\}$ is a set of representatives for $\mc{O}^2/\rho \mc{O}^2$, then $\{(x_1,x_2)\varepsilon M\}$ and $\{(x_1+\lambda,x_2+\mu)\}$ are again a set of representatives for $\mc{O}^2/\rho \mc{O}^2$. Further using that $N(\rho)|p$ and the formulas (\ref{eq:gpact1}), (\ref{eq:gpact2}) we get,
\begin{align*}
\left(\phi|u_\rho\right)|_{k,m/N(\rho)}\varepsilon M&=\left(\phi|_{k,m}\varepsilon M\right)|u_\rho=\phi|u_\rho.\\
(\phi|u_\rho)\mid_{m/N(\rho)}[\lambda,\mu]&=\phi|u_\rho.
\end{align*}
To complete the proof we find the Fourier expansion of $\phi|u_\rho$. Let $x=(x_1,x_2)\in \mc{O}^2/\rho \mc{O}^2$, then from (\ref{eq:gpact2})
\begin{align*}
\left(\phi|\left[\tfrac{x}{\rho}\right]\right)(\tau,z_1/\rho,z_2/\overline{\rho})=e\big(\tfrac{pN(x_1)}{N(\rho)}\tau+\tfrac{p\overline{x_1}}{\overline{\rho}}z_1+\tfrac{px_1}{\rho}z_2)\phi\big(\tau,z_1+\tfrac{x_1}{\rho}\tau+\tfrac{x_2}{\rho},z_2+\tfrac{\overline{x_1}}{\overline{\rho}}\tau+\tfrac{\overline{x_2}}{\overline{\rho}}\big).
\end{align*}
On writing the Fourier expansion and using \eqref{expo}, we see that $\phi|u_\rho$ equals
\begin{align*}
\underset{n,r}{\sum_{x_1\in \mc{O}/\rho \mc{O}}}c_\phi\left(n-Re\big((\tfrac{2ir}{\D}-\tfrac{m\overline{x}_1}{\rho\overline{\rho}})x_1\big), \rho\left(r-\tfrac{\D m\overline{x}_1}{i\rho\overline{\rho}}\right)\right)e(n\tau+\tfrac{ir}{\D}z_1+\tfrac{\overline{ir}}{\D}z_2).
\end{align*}
Now let $r'=r-\frac{\D m\overline{x}_1}{i\rho\overline{\rho}}$. As $x_1$ varies modulo $\rho$, $r'$ varies modulo $\frac{i\D m}{\rho}$ with $r'\equiv r\text{ (mod}\frac{i\D m}{\rho\overline{\rho}})$. Also we have that $\frac{2ir}{\D}-\frac{m\overline{x}_1}{\rho\overline{\rho}}=\frac{i}{\D}\left(r'+r\right)$ and so,
\begin{equation}\label{eq:up}
\phi|u_\rho=\sum_{n,r}\underset{r'\equiv r\text{ (mod}\frac{i\D m}{\rho\overline{\rho}})}{\sum_{r'\text{ (mod}\frac{i\D m}{\rho})}}c_\phi\left(n+N(\rho)\tfrac{N(r')-N(r)}{|D|m},\rho r'\right)e(n\tau+\tfrac{ir}{\D}z_1+\tfrac{\overline{ir}}{\D}z_2).
\end{equation}
From this the conditions at cusps are easily seen to be satisfied. This completes the proof.
\end{proof}

\begin{prop}\label{prop:ucombi}
Let $\rho\in \mc{O}$.
\begin{enumerate}
\item[(a)] If $\phi\in \mc{J}_{k,m}$, then $\phi|U_\rho u_\rho=N(\rho)\phi$.
\item[(b)] If $\phi\in \mc{J}_{k,1}$ and $(\rho,\overline{\rho})=1$, then $\phi|U_\rho u_{\overline{\rho}}=\phi$.
\end{enumerate}
\end{prop}
\begin{proof}
$(a)$. Let $c_{\rho\rho}(n,r)$ denote the $(n,r)$-th Fourier coefficient of $\phi|U_\rho u_\rho$. Then from (\ref{eq:Up}) and (\ref{eq:up}) we have
\begin{equation*}
c_{\rho\rho}(n,r)=\underset{r'\equiv r\text{ (mod }i\D m)}{\sum_{r'\text{ (mod }i\D\overline{\rho} m)}}c_\phi\left(n+\tfrac{N(r')-N(r)}{|D|m},r'\right)=N(\rho)c_\phi(n,r).
\end{equation*}
The last step follows from the fact that if $\phi\in \mc{J}_{k,m}$, then $c_\phi(n',r')=c_\phi(n,r)$ whenever $|D|n'm-N(r')=|D|nm-N(r)$ and $r'\equiv r\pmod{i\D m}$. This condition is satisfied in each summand above.

$(b)$. Let $c_{\rho\overline{\rho}}(n,r)$ denote the $(n,r)$-th Fourier coefficient of $\phi|U_\rho u_{\overline{\rho}}$. Then we have 
\begin{equation*}
c_{\rho\overline{\rho}}(n,r)=\underset{r'\equiv r\text{ (mod }i\D)}{\sum_{r'\text{ (mod }i\D\rho )}}c_\phi\left(n+\tfrac{N(r')-N(r)}{|D|m},\tfrac{\overline{\rho}r'}{\rho}\right).
\end{equation*}
Since $(\rho,\overline{\rho})=1$, the only non-zero summand is for which $\rho|r'$. But there exists exactly one such $r'$ $\text{(mod }i\D\rho) $ with $r'\equiv r\text{ (mod }i\D)$. Now the proof follows by noting that if $\phi\in\mc{J}_{k,1}$, then $c_\phi(n',r')=c_\phi(n,r)$ whenever $|D|n'-N(r')=|D|n-N(r)$ and $r'\equiv r\pmod{i\D}$.
\end{proof}
Let $\mc{J}_{k,m}^{spez}$ denote the subspace of $\mc{J}_{k,m}$ consisting of those $\phi\in\mc{J}_{k,m}$ whose Fourier coefficients $c(n,r)$ depend only on $|D|nm-N(r)$. We present the following arguments for the benefit of the reader.

\begin{prop}\label{prop:injspez}
The Eichler-Zagier map $\iota$ defined in section \ref{sec:introhjf} is injective on $\mc{J}_{k,m}^{spez}$.
\end{prop}
\begin{proof}
Let $\phi\in \mc{J}_{k,m}^{spez}$. Then $c_\phi(n,r)=c_\phi(n',r')$, whenever $|D|n'm-N(r')=|D|nm-N(r)$. Recall from (\ref{def:hs}), the definition of the theta component $h_s$. \[h_{s}(\tau)=\underset{N(s)+n\in |D|m\mathbf{Z}}{\sum_{n>0}}c_\phi\big(\tfrac{n+N(s)}{|D|m},s\big)e(n\tau /|D|m).\]
But $|D|\big(\tfrac{n+N(s)}{|D|m}\big)m-N(s)=n$, thus $c_\phi\big(\tfrac{n+N(s)}{|D|m},s\big)=c_\phi(\tfrac{n}{|D|m},0)$. That is $h_s=h_0$. This is true for every $s$. Now if $h:=\iota(\phi)=0$, then $0=h(\tau)=m|D|i h_0\mid_{k-1} J(m|D|\tau)$ (from (\ref{hinversion})). Thus $h_0=0$. This along with $h_s=h_0$ for all $s$ implies $\phi=0$.
\end{proof}

\begin{lem}\label{lemjspez}
\text{ }
\begin{enumerate}
\item[(a)] For $k\neq 0\pmod {w(D)}$, $\mc{J}_{k,m}^{spez}={0}.$
\item[(b)] For $k=0\pmod {w(D)}$, $\mc{J}_{k,1}^{spez}=\mc{J}_{k,1}.$
\end{enumerate}
\end{lem}
\begin{proof}
For $(a)$ note that from (\ref{eq:units}), we have $\varepsilon^k h_0=h_0$, for $\varepsilon\in O^\times$. Since $k\neq 0\pmod {w(D)}$, choosing suitable $\varepsilon$ we get $h_0=0$. Thus $h=\iota(\phi)=0$ for any $\phi\in\mc{J}_{k,m}$. Now the proof follows from proposition \ref{prop:injspez}.

$(b)$ follows from the fact that if $\phi\in\mc{J}_{k,1}$, then $c_\phi(n',r')=c_\phi(n,r)$ whenever $|D|n'-N(r')=|D|n-N(r)$ and $r'\equiv r\pmod{i\D}$. Moreover for our choice of discriminants $D$ (which are of the form $-p$, $p \equiv 3 \bmod 4$ as in \thmref{th:01}), $N(r)-N(r')\in |D|\mbf \cdot \mbf Z$ implies that $\varepsilon\in O^\times$ such that $r-\varepsilon r'\in i\D \mc O$. This can be checked by hand for $D=-4,-8$ and for odd $D$, using \lemref{lem:coset}.
\end{proof}

\begin{lem}\label{lem:phiup}
Let $\rho\in \mc{O}$ and $\phi\in \mc{J}_{k,m}^{spez}$. Then $\phi\left|U_\rho\right.\in \mc{J}_{k,mN(\rho)}^{spez}$.
\end{lem}
\begin{proof}Let $c_{\rho}(n,r)$ denote the $(n,r)$-th Fourier coefficient of $\phi|U_\rho$. Then $c_\rho(n,r)=0$ if $\rho\nmid r$ and $c_\rho(n,r)=c_\phi(n,r/\rho)$ if $\rho|r$. Hence it is enough to prove the result for $(n,r)$ when $\rho|r$.

Let $(n,r)$ and $(n',r')$ be such that $|D|nmN(\rho)-N(r)=|D|n'mN(\rho)-N(r')$. Then $|D|nm-\frac{N(r)}{N(\rho)}=|D|n'm-\frac{N(r')}{N(\rho)}$. i.e., we have $c_\phi(n,r/\rho)=c_\phi(n',r'/\rho)$. Since $\phi\in \mc{J}_{k,m}^{spez}$ this implies $c_\rho(n,r)=c_\rho(n',r')$. Thus $\phi\left|U_\rho\right.\in J_{k,mN(\rho)}^{spez}$.
\end{proof}
For any $l\in \mbf{N}$, like in classical case we can define an operator $V_l:\mc{J}_{k,m}\longrightarrow\mc{J}_{k,ml}$ (see \cite{haverkamp}). For any $\phi\in\mc{J}_{k,m}$, the Fourier expansion of $\phi\left|V_l\right.$ is given by
\begin{equation}
\phi\left|V_l\right.(\tau,z_1,z_2)=\sum_{n\ge 0}\sum_{N(r)\le |D|lmn}\Big(\underset{r/a\in \mc{O}}{\sum_{a|(n,l)}}a^{k-1}c_\phi(\tfrac{nl}{a^2},\tfrac{r}{a})\Big)e(n\tau+\tfrac{ir}{\D}z_1+\tfrac{\overline{ir}}{\D}z_2).
\end{equation}

\subsection{Injectivity of Eichler-Zagier map $\iota$} \label{inj}
The aim of this subsection is to indicate that the Eichler-Zagier map $\iota:\mc{J}_{k,p}\longrightarrow S_{k-1}(|D| p,\chi_{D})$ defined by $\phi \mapsto \iota (\phi)=:h$ (as in \eqref{eq:h}), may fail to be injective at least for certain primes $p$. This is in contrast with the classical case where it is known (see \cite{ez}) that the Eichler-Zagier map is injective for prime indices. Perhaps this subsection justifies our efforts in section~\ref{hjf} to prove \thmref{th:01} using these maps. In this subsection we restrict ourselves to $K=\mbf{Q}(i)$ (i.e., $D=-4$). We start with some auxiliary results.

\begin{lem}\label{lem:vpinjective}
For any odd prime $p \in \mbf N$, $p>5$, let $V_p$ be the operator on $\mc{J}_{k,1}$. Then $V_p$'s are injective on $\mc{J}_{k,1}^{cusp}$.
\end{lem}
\begin{proof}
Let $V^*_p$ denote the adjoint of $V_p$. Then we have from \cite[p.190]{raghavan}, $V^*_pV_p=T_p+(p+1)p^{k-2}$, where $T_p$ is the $p$-th Hecke operator on $\mc{J}_{k,1}$. Suppose $\phi\in\mc{J}_{k,1}^{cusp}$ is such that $\phi\left|V_p\right.=0$, write $\phi=\sum c_i\phi_i$ as a sum of Hecke eigenforms, say, with $c_1>0$. Then we get that $\lambda_1(p)=-(p+1)p^{k-2}$, where $\phi_1|T_p=\lambda_1(p)\phi_1$. We also have from \cite[Lemma 2, p.195]{raghavan} that $\lambda_1(p)=a(p^2)-p^{k-3}\chi_{-4}(p)$ for any odd prime $p$ and for some normalized eigenform $f\in S_{k-1}(\Gamma_0(4),\chi_{-4})$ such that $f(\tau)=\tsum_{n\ge 1}a(n)e(n\tau)$. This means that $a(p^2)=-p^{k-1}-p^{k-2}+p^{k-3}\chi_{-4}(p)$. Thus
\begin{equation*}
|a(p^2)|=|p^{k-1}+p^{k-2}-p^{k-3}\chi_{-4}(p)|=p^{k-1}|1+\tfrac{1}{p}-\tfrac{1}{p^2}\chi_{-4}(p)|> p^{k-1}.
\end{equation*}
But this is impossible since we have $|a(p^2)|\le 3p^{k-2}$ (from Deligne's bound). Thus $V_p$ must be injective on $\mc{J}_{k,1}^{cusp}$.
\end{proof}

\begin{prop} \label{spez}
$\mc{J}_{k,p}\setminus\mc{J}_{k,p}^{spez}$ is non-zero when $k \geq 12$ is even, and $p>5$ splits in $\mf Q(i)$.
\end{prop}
\begin{proof}
For $k,p$ as in the theorem, we claim that there exist a non-zero $\Phi\in\mc{J}_{k,1}^{cusp}$ such that $ \Phi \vert V_p\notin\mc{J}_{k,p}^{spez}$. Note that $\Phi \vert V_p \neq 0$ by \lemref{lem:vpinjective}. To prove this, we start more generally by taking a non-zero form $\phi \in \mc J_{\kappa,1}$ ($\kappa >4$) and consider $\phi \vert V_p$.

Now $c_p(n,r)=c_\phi(np,r)+p^{\kappa -1}c_\phi(\frac{n}{p},\frac{r}{p})$, where $c_p(n,r)$ is the $(n,r)$-th Fourier coefficient of $\phi|V_p$ and the term $c_\phi(\frac{n}{p},\frac{r}{p})=0$ if either $p\nmid n$ or $p\nmid r$. Since $p$ splits in $\mbf{Q}(i)$, we can write $p=\pi\overline{\pi}$, where $\pi\in\mc{O}$ is a prime. Choose two pairs of $(n,r)$ as $n_1=Np$, $r_1=p$ and $n_2=Np$, $r_2=\pi^2$. Then $4n_1p-N(r_1)=4n_2p-N(r_2)$. But $c_p(n_1,r_1)=c_\phi(Np^2,p)+p^{\kappa -1}c_\phi(N,1)$ and $c_p(n_2,r_2)=c_\phi(Np^2,\pi^2)$. Since $\phi\in\mc{J}_{8,1}$ and from lemma \ref{lemjspez}, we get $c_\phi(Np^2,p)=c_\phi(Np^2,\pi^2)$. Thus to prove our claim, it is enough to get a $\phi$ such that $c_\phi(N,1)\neq 0$ for some $N>0$ or equivalently $h_1(\phi)\neq 0$, where $h_1(\phi)$ denotes the `odd' theta component of $\phi$.

Let $\Psi:=\Psi_{8,1} \in \mc J_{8,1}$ be the cusp form as given in \cite[p.~308]{sasaki}. Then one can directly verify that the theta component $ h _1(\Psi)$ of $\Psi$ is non-zero. Now consider the Jacobi form
\[  \Psi_k = E_k \cdot \Psi,  \]
where $E_k \in M^1_k$ is the Eisenstein series in one variable. Clearly $\Psi_k \in J_{k+8,1}$ is such that $ h_1(\Psi_k) = E_k \cdot  h _1(\Psi) \neq 0$. By our discussion in the above paragraph (with $\phi=\Psi_k$ and $\kappa=k+8 \geq 12$), we see that $\Psi_k|V_p \not \in \mc{J}_{k,p}^{spez}$.
\end{proof}

\begin{prop} \label{pnonsplit}
If $p$ does not split in $\mf Q(i)$, then $\mc{J}_{k,p}^{spez}$ is the maximal subspace of $\mc{J}_{k,p}$ on which the Eichler-Zagier map $\iota$ is injective.
\end{prop}

\begin{proof}
We first claim that, under the above assumptions, $\mc J^{\eta_0}_{k,p} = \mc{J}_{k,p}^{spez}$. Granting this for the moment, note that the proposition follows since $\iota$ annihilates $\mc{J}_{k,p} \setminus \mc{J}_{k,p}^{spez}$; see section~\ref{sec:decomp}. To prove the above equality, by the same reason as above, clearly $\mc{J}_{k,p}^{spez} \subseteq \mc J^{\eta_0}_{k,p} $. 

Now suppose that $\phi \in \mc J^{\eta_0}_{k,p}$, so that $h_{\mu s} = h_s$ for all $s \bmod 2p$ such that $(s,2p)=1$ and $\mu \in G$ (see \eqref{Gdef}). Therefore it is enough to show that $r_1 \equiv \mu r_2 \bmod 2p$ for some $\mu \in G$, whenever $r_1,r_2 \in \mc O$, with $N(r_1) \equiv N(r_2) \bmod 4p$ and $(r_1r_2,2p)=1$. The proof now is a easy exercise in congruences, and we omit it.
\end{proof}

\begin{rmk} \label{e-z}
Summarizing the content of the above results, we see that in general $\mc{J}_{k,p}^{spez}$ could be strictly smaller than $\mc J_{k,p}$ and that $\iota$ may fail to be injective in its complement.
\end{rmk}

\subsection{Index-old Hermitian Jacobi forms of index $p$}\label{sec:indexold}
In this subsection we prove the assumptions made in the section~\ref{subsec:thm1} that given $\phi\in \mc{J}_{k,p}$, either $h\neq 0$ or $h_{\tilde{\eta}}\neq 0$ with $\eta$ and $\tilde{\eta}$ as in section \ref{sec:introhjf}. In the process we also show that if $\phi\in \mc{J}_{k,p}$ is such that $c(n,s)=0$ for all $s$ with $(s,i\D p)=1$, then either $\phi=0$ or $\phi$ must come from a Hermitian Jacobi form of lower index depending on whether $\chi_D(p)=-1$ or $\chi_D(p)=1$ respectively.

Let $G$ be the group defined in section \ref{sec:introhjf}. Denote the group $(\mc{O}/i\D p\mc{O})^\times$ by $\tilde G$.

\begin{prop}\label{prop:hs=0}
Let $\phi\in \mc{J}_{k,p}$ be such that $h_{\tilde{\eta}}=0$ for all extensions $\tilde{\eta}$ of any character $\eta$ of $G$. Then the theta components $h_s$ of $\phi$ are zero for all $(s,i\D p)=1$.
\end{prop}
\begin{proof}
Suppose for any character $\eta$ on $G$
\[  h_{\tilde{\eta}} = \sum_{t \in \mc{O}/i\D p\mc{O}} \overline{  \tilde{\eta}(t)  } h_t =  0 \q \text{for all extensions } \tilde{\eta} \text{ of } \eta. \]

Let us fix $\delta$ to be one character which extends $\eta$. We say that $\delta$ is over $\eta$. Then all other characters which extend $\eta$ are of the form
$\delta \cdot \widehat{ \tilde G/G}$, where $\text{ } {\widehat{ }}\text{ }$ denotes the character group. Let $s\in \mc{O}$ with $(s,i\D p)=1$. Now look at the sum
\[  \sum_{ \tilde{\eta} \text{ over } \eta } \tilde{\eta}(s)  h_{\tilde{\eta}}  = \sum_{t \in \mc{O}/i\D p\mc{O}} \left(  \sum_{\lambda}   \overline{  \delta \lambda(ts^{-1})  } \right)  h_t.\]

\noindent In the above sum, $\lambda$ varies in $\widehat{ \tilde G/G  }$. Let us look at the sum in braces. Let $\alpha \in \mc{O}$. Then by orthogonality
\[  \sum_\lambda (\delta \lambda) (\alpha) = \delta(\alpha) \sum_\lambda \lambda(\alpha) = \begin{cases}  0, &\text{ if } (\alpha, i\D p) \neq 1; \\ \delta(\alpha) \# (\tilde{G}/G), &\text{ if } \alpha \in G;\\ 0, &\text{ if } \alpha  \in \tilde G - G.\end{cases}  \]

This means that the sum above is
\begin{equation}\label{eq:firstcharsum} 
0= \sum_{ \tilde{\eta} \text{ over } \eta } \tilde{\eta}(s)  h_{\tilde{\eta}}  =  \# (\tilde{G}/G) \underset{\mu \in G}\sum \overline{ \eta( \mu ) } h_{\mu s}.
\end{equation}

Note that when $\eta(\epsilon)\neq \epsilon^{-k}$, then $h_{\tilde{\eta}}$ is automatically zero (see \cite{haverkamp}). Thus (\ref{eq:firstcharsum}) is true for all characters $\eta$ on $G$. Now sum (\ref{eq:firstcharsum}) over characters of $G$ to get
\[ 0= \sum_{ \eta \in \widehat{G} } \underset{\mu \in G}\sum \overline{ \eta( \mu ) } h_{\mu s} = \# (G)h_s.\] Thus $h_s=0$ for all $(s,i\D p)=1$. This completes the proof.
\end{proof}

\begin{rmk}
Proposition \ref{prop:hs=0} remains true for any fundamental discriminant $D$.
\end{rmk}

\begin{lem}\label{lem:coset}
Let $\rho\in \mc{O}$ be a prime such that $N(\rho)\in \mbf{Z}$ is a prime. Then $\{1,2,......,N(\rho)\}$ is a set of coset representatives for $\mc{O}/\rho \mc{O}$.
\end{lem}
\begin{proof}
Let $\mc{C}=\{1,2,......,N(\rho)\}$. It is enough to prove that any two distinct elements of $\mc{C}$ are not congruent modulo $\rho$. Suppose $\alpha,\beta\in \mc{C}$ are such that $\alpha\equiv\beta\pmod \rho$. Since $N(\rho)$ is a prime, this would imply $N(\rho)|(\alpha-\beta)$. But this is possible only when $\alpha=\beta$. 
\end{proof}

\begin{rmk}\label{rmk:oddprime}
Note that when $D$ is odd, $|D|$ is a prime in $\mbf{Z}$. Thus $i\D$ is a prime in $\mc{O}$. If $D=-4$, then $i\D=(1+i)^2$ and if $D=-8$, then $i\D=(-i\sqrt{2})^3$. Both $(1+i)$ and $-i\sqrt{2}$ are primes in their respective ring of integers.
\end{rmk}

\begin{prop}\label{prop:p3mod4s}
Let $p\in \mbf{Z}$ be a prime such that $\chi_{D}(p)=-1$ and $\phi\in \mc{J}_{k,p}$ be such that $h_s=0$ for $s\in \mc{O}$ with $(s,i\D p)=1$. Then $h_s=0$ for $(s,i\D)=1$.
\end{prop}
\begin{proof}
First we consider the case when $D$ is odd. Since $p$ is a prime in $\mc O$ and $(p,D)=1$ and since we already know that $h_s=0$ for $(s,i\D p)=1$, it is enough to prove that $h_s=0$ for $s$ with $(s,i\D p)=p$. Any such $s$ is of the form $\alpha p$, where $\alpha\in \mc{O}/i\D \mc{O}$ and $(\alpha, i\D)=1$. By remark \ref{rmk:oddprime} and lemma \ref{lem:coset}, we can choose $\mc{C}=\{1,2,......|D|\}$ to be the set of coset representatives for $\mc{O}/i\D \mc{O}$. Now if $N(\alpha p)\equiv N(\beta p) \pmod {|D|p}$ for some $\alpha,\beta\in \mc{C}$, then we must have that $N(\alpha)\equiv N(\beta)\pmod{|D|}$ or equivalently $\alpha^2\equiv\beta^2\pmod {|D|}$. But since $|D|$ is a prime we must have that $\alpha=\pm \beta$. 

From the given condition and using (\ref{hinversion}) we get for any $s$ with $(s,i\D p)=1$,
\begin{equation*}
\sum_{r\in \mc{O}/i\D p\mc{O}}e\left(\tfrac{2Re(s\overline{r})}{|D|p}\right)h_r=0.
\end{equation*}
Now it is clear from (\ref{eq:units}) that $h_{\alpha p}=0$. This completes the proof when $D$ is odd.

When $D=-4$ or $-8$, we replace $i\D$ in the above proof by $(1+i)$ and $-i\sqrt{2}$ respectively and proceed to prove the result in the same manner.
\end{proof}

\begin{prop}\label{prop:p3mod4}
Let $\phi\in \mc{J}_{k,p}$ be such that $c(n,s)=0$ for all $s$ with $(s,i\D)=1$. Then $\phi=0$.
\end{prop}
\begin{proof}
First we prove this for the case $D$ odd. Let $\phi(\tau,z_1,z_2)=\sum c(n,s)e(n\tau+\tfrac{is}{\D}z_1+\tfrac{\overline{is}}{\D}z_2)$. For any $r\in \mc{O}$ with $(r,i\D)=i\D$, we have $1-e\left(\tfrac{2}{\D}Re\left(\tfrac{r}{\D}\right)\right)=0$. Since $c(n,s)=0$ for all $s$ with $(s,i\D)=1$, for any $r\in \mc{O}$ we have
\begin{equation*}
\left(\phi-\phi\left|\left[0,\tfrac{r}{i\D}\right]\right)(\tau,z_1,z_2)\right.=\sum c(n,s)\left(1-e\left(\tfrac{2}{\D}Re\left(\tfrac{rs}{\D}\right)\right)\right)e(\cdots)=0.
\end{equation*}
Now applying the matrix $\begin{psmallmatrix}* & 0\\1 & 1\end{psmallmatrix}\in \Gamma_1(\mc{O})$ to the above equation and using the formulas in section \ref{sec:introhjf} we get,
\begin{equation}\label{eq:notone}
\phi=e\left(\tfrac{pN(r)}{|D|}\right)\phi\left|\left[\tfrac{r}{i\D},\tfrac{r}{i\D}\right]\right..
\end{equation}
Also applying the matrix $\begin{psmallmatrix}0 & -1\\1 & 0\end{psmallmatrix}\in \Gamma_1(\mc{O})$ we get, $\phi=\phi\left|\left[\tfrac{r}{i\D},0\right]\right.$. Thus
\begin{equation}\label{eq:one}
\phi=\phi\left|\left[0,\tfrac{r}{i\D}\right]\right.\left|\left[\tfrac{r}{i\D},0\right]\right.=\phi\left|\left[\tfrac{r}{i\D},\tfrac{r}{i\D}\right]\right..
\end{equation}
Now from (\ref{eq:notone}) and (\ref{eq:one}) and from the fact that $(p,|D|)=1$, we get $\phi=0$.

When $D=-4$ or $-8$, the proof follows similarly by replacing $i\D$ by $(1+i)$ and $-i\sqrt{2}$ respectively.
\end{proof}

\begin{cor}
Let $p\in \mbf{Z}$ be a prime such that $\chi_{D}(p)=-1$ and $\phi\in \mc{J}_{k,p}$ be non-zero. Then there exists a character $\eta$ of $G$ such that $h_{\tilde{\eta}}\neq 0$.
\end{cor}
\begin{proof}
This is immediate from propositions \ref{prop:hs=0}, \ref{prop:p3mod4s} and \ref{prop:p3mod4}.
\end{proof}

For any $r\in \mc{O}$, let $(r)$ denote the ideal generated by $r$. We now define the M\"{o}bius Function on $\mc{O}$ similarly as in case of $\mbf{Z}$.
\begin{defi}
Let $r\in \mc{O}$, then define the M\"obius function $\mu$ as follows
\[\mu(r)=\begin{cases}
1, & \text{ when } (r)=(1);\\
(-1)^t, & \text{ if } (r)=\mathfrak{p}_1\mathfrak{p}_2\cdots\mathfrak{p}_t \text{ for distinct prime ideals } \mathfrak{p}_i;\\
0, & \text{ otherwise}.
\end{cases} \]
\end{defi}

The following lemma is the starting point of our discussion of index old forms.
\begin{lem}\label{lem:expsum}
Let $r,s\in \mc{O}$ be such that $(r,s)\neq 1$. Then $\underset{t|s}{\sum}\mu(t)\underset{\pi|t}{\prod}e\left(\tfrac{2}{\D}Re\left(\frac{ir}{\pi}\right)\right)=0$.
\end{lem}
\begin{proof}
By our assumption on $r$ and $s$, there exists a prime divisor $\pi$ of $s$ such that $\pi|r$. Thus we have $1-e\left(\tfrac{2}{\D}Re\left(\frac{ir}{\pi}\right)\right)=0$. Now taking the product over all prime divisors of $s$ we get
\begin{equation*}
\prod_{\pi|s}\left(1-e\left(\tfrac{2}{\D}Re\left(\frac{ir}{\pi}\right)\right)\right)=0.
\end{equation*}
Expanding the product on the left hand side we get the required expression.
\end{proof}

\begin{prop}\label{prop:p1mod4}
Let $p\in \mbf{Z}$ be a prime such that $\chi_{D}(p)=1$ and $\phi\in \mc{J}_{k,p}$ be such that $h_s=0$ for $s\in \mc{O}$ with $(s,i\D p)=1$. Then $\phi\in \mc{J}_{k,1}|U_{\pi}+\mc{J}_{k,1}|U_{\overline{\pi}}$, where $p=\pi\overline{\pi}$, with $\pi\in\mc O$ prime.
\end{prop}
\begin{proof}
Suppose $\phi$ is such that $h_s=0$ for all $s\in \mc{O}$ with $(s,i\D p)=1$. Then $c(\tfrac{n+N(s)}{|D|p},s)=0$ for all $s\in \mc{O}$ with $(s,i\D p)=1$. This in turn implies $c(n,s)=0$ whenever $(s,i\D p)=1$. We now prove the proposition by adapting a method as outlined in \cite{skoruppa-zagier}, and with some care.

We first prove the result when $D$ is odd. Let $r\in \mc{O}$, then using $c(n,s)=0$ whenever $(s,i\D p)=1$ and lemma \ref{lem:expsum} with $s=i\D p$ we get
\begin{equation}\label{eq:expsum}
\underset{t|i\D p}{\tsum}\mu(t)\phi \left| \Big(\right.\underset{\rho|t}{\tprod}\left[0,\tfrac{r}{\rho}\right]\Big)=0.
\end{equation}
Now by applying suitable matrices $\begin{psm}
*&*\\c&d
\end{psm}\in\Gamma_1(\mc{O})$ and summing up we get,
\begin{align}
0&=\sum_{r=1}^{|D|p^2}\underset{(c,d)=1}{\sum_{c,d=1}^{\lf \frac{|D|p^2}{r}\rf}}\underset{t| i\D p}{\sum}\mu(t)\phi\left| \left(\underset{\rho|t}{\prod}\left[\left(\tfrac{rc}{\rho},\tfrac{rd}{\rho}\right),e\big(\tfrac{cdr^2}{N(\rho)}\big)\right]\right)\right.\nonumber\\ 
&=\underset{t|i\D p}{\sum}\mu(t)\sum_{x_1,x_2=1}^{|D|p^2}\phi\left| \left(\underset{\rho|t}{\prod}\left[\left(\tfrac{x_1}{\rho},\tfrac{x_2}{\rho}\right),e\big(\tfrac{x_1\overline{x_2}}{N(\rho)}\big)\right]\right)\right.\nonumber
\end{align}
Recall that $\rho$ in the above sums are primes in $\mc O$ and that $N(\rho)$ is a prime in $\mbf Z$ (cf.~\rmkref{rmk:oddprime}). Now using \lemref{lem:coset} we note that for each $\rho$, the number of distinct $x_1\text{(mod }\rho)$ as $x_1$ varies from $1$ to $|D|p^2$ is $|D|p^2/N(\rho)$ (and similarly for $x_2$). Finally using the Chinese Remainder Theorem  we can rewrite the above as
\begin{equation}\label{eq:skzaop}
0=\underset{t|i\D p}{\sum}\mu(t)\tfrac{|D|^2p^4}{N(t)^2}\underset{\rho|t}{\prod}\sum_{x_1,x_2\text{ (mod }\rho)}\phi\left|\left[\left(\tfrac{x_1}{\rho},\tfrac{x_2}{\rho}\right),e\big(\tfrac{x_1\overline{x_2}}{N(\rho)}\big)\right]\right..
\end{equation}

Since \eqref{eq:expsum} is unchanged if we change the order of the variables involved, separating the terms in (\ref{eq:skzaop}) according to $(t,i\D)=1$ or not (and recalling that $i\D$ is a prime) we get,
\begin{equation}\label{eq:psi}
|D|^2\psi=\sum_{x_1,x_2=1}^{|D|}\psi\left| \left[\left(\tfrac{x_1}{i\D},\tfrac{x_2}{i\D}\right),e\big(\tfrac{x_1\overline{x_2}}{|D|}\big)\right]\right. ,
\end{equation}
where
\begin{equation*}
\psi=p^4\phi-p^2\sum_{x\text{ (mod }\pi)}\phi\left|\left[\tfrac{x}{\pi}\right]\right.-p^2\sum_{x\text{ (mod }\overline{\pi})}\phi\left|\left[\tfrac{x}{\overline{\pi}}\right]\right.+\underset{\pi|p}{\prod}\sum_{x\text{ (mod }\pi)}\phi\left|\left[\tfrac{x}{\pi}\right]\right. .
\end{equation*}

Since $\chi_{D}(p)=1$, $p$ splits in $\mc{O}$, say $p=\pi\overline{\pi}$. Now using the operators defined in section \ref{subsec:operators} we have,
\begin{align*}
\psi&=p^4\phi-p^3\phi\left|u_\pi U_\pi\right.-p^3\phi\left|u_{\overline{\pi}} U_{\overline{\pi}}\right.+p^2\phi\left|u_{\overline{\pi}}U_{\overline{\pi}}u_{\pi} U_{\pi}\right. \\
&=p^4\phi-p^3\phi\left|u_\pi U_\pi\right.-p^3\phi\left|u_{\overline{\pi}} U_{\overline{\pi}}\right.+p^2\phi\left|u_{\overline{\pi}}U_{\pi}\right. .
\end{align*}
The last equality follows from part (b) of proposition \ref{prop:ucombi}. 
Thus $\psi\in\mc{J}_{k,p}$. 

Now applying $[0,r]\in \mc{O}^2$ to (\ref{eq:psi}) we get
\begin{equation*}
|D|^2\psi=\sum_{x_1,x_2}e(2pRe\big(\tfrac{x_1r}{i\D}\big))\psi\left| \left[\left(\tfrac{x_1}{i\D},\tfrac{x_2}{i\D}\right),e\big(\tfrac{x_1\overline{x_2}}{|D|}\big)\right]\right. .
\end{equation*}
Now summing over $r \pmod {i\D}$ in (\ref{eq:psi}) and using \eqref{expo} for the exponential sum over $r$, we get 
\begin{equation*}
|D|^3\psi=|D|\sum_{x_2=1}^{|D|}\psi\left| \left[0,\tfrac{x_2}{i\D}\right]\right. .
\end{equation*}
Now writing the Fourier expansion of r.h.s, we find that $c_\psi(n,s)=0$ for all $s$ with $(s,i\D)=1$. That is, $\psi$ satisfies the hypothesis of proposition \ref{prop:p3mod4}. Thus $\psi=0$, that is we get
\begin{equation*}
\phi-\tfrac{1}{p}\phi\left|u_\pi U_\pi\right.-\tfrac{1}{p}\phi\left|u_{\overline{\pi}} U_{\overline{\pi}}\right.+\tfrac{1}{p^2}\phi\left|u_{\overline{\pi}} U_{\pi}\right.=0.
\end{equation*}
When $D=-4$ or $-8$, the proof follows very similarly by replacing $i\D$ by $(1+i)$ ad $-i\sqrt{2}$ respectively. We omit the details. This completes the proof since $\phi|u_{\pi},\phi\left|u_{\overline{\pi}}\right.\in \mc{J}_{k,1}$.
\end{proof}

\begin{cor}
Let $p\in \mbf{Z}$ be a prime such that $\chi_D(p)=1$ and $\phi\in \mc{J}_{k,p}$ be non zero.
\begin{enumerate}
\item[(a)] If $k\not\equiv 0\pmod {w(D)}$, then there exists a character $\eta$ over $G$ such that $h_{\tilde{\eta}}\neq 0$.
\item[(b)] If $k\equiv 0\pmod {w(D)}$ and $h_{\tilde{\eta}}=0$ for all $\tilde{\eta}$, then $h\neq 0$.
\end{enumerate}
\end{cor}
\begin{proof}
(a). Suppose not, then by proposition \ref{prop:p1mod4} $\phi\in \mc{J}_{k,1}|U_{\pi}+\mc{J}_{k,1}|U_{\overline{\pi}}$. By lemma~\ref{lem:phiup} this means $\phi\in \mc{J}_{k,p}^{spez}=\{0\}$ (by Lemma \ref{lemjspez}), a contradiction.

\noindent(b). By proposition \ref{prop:hs=0}, the given condition means $h_s=0$ for all $(s,i\D p)=1$. Thus by proposition \ref{prop:p1mod4} and \ref{lem:phiup}
we have $\phi\in \mc{J}_{k,p}^{spez}$. Thus $h\neq 0$ (see proposition \ref{prop:injspez}).
\end{proof}

\begin{proof}[\textbf{Proof of \propref{prop:finalprop}}]
This is immediate from the above corollary.
\end{proof}

\begin{rmk} \label{classno1}
When the class number of $\mbf{Q}(\sqrt{D})$ is not $1$, we do not see immediately how to adapt the arguments used in \propref{prop:p1mod4}. Moreover, the definition of the operator $U_\rho$ ($\rho \in \mc{O}$) perhaps has to be generalised to the setting of ideals, which again is not clear at the moment.
\end{rmk}

\subsection{Some lemmas about characters of $G$.}\label{subsec:lemmaG}
To descend to the elliptic modular forms we must control $N/m_\chi$ ratio. To this end we prove the following results about the characters of $G$ defined as in section \ref{sec:introhjf}.
\begin{lem}\label{lem:nontrivialeta}
Let $\eta$ be a character of $G$. Then there exists an extension $\tilde{\eta}$ of $\eta$ to $\tilde{G}$ such that restriction of $\tilde{\eta}$ to $\mbf{Z}$ is non trivial and its conductor is divisible by $p$.
\end{lem}
\begin{proof}
Since $(i\D,p)=1$, any extension $\tilde{\eta}$ of $\eta$ to $\tilde{G}$ can be decomposed as $\tilde{\eta}=\tilde{\eta}_D\cdot\tilde{\eta}_p$, where $\tilde{\eta}_D$ and $\tilde{\eta}_p$ are characters of $(\mc{O}/i\D \mc{O})^\times$ and $(\mc{O}/p\mc{O})^\times$ respectively.

Let $\psi$ denote the restriction of $\tilde{\eta}$ to $\mbf Z$ so that $\psi$ is a Dirichlet character mod $|D|p$. Since $(|D|,p)=1$ we can decompose $\psi=\psi_{|D|}\cdot\psi_p$, where $\psi_{|D|}$ and $\psi_p$ are Dirichlet characters mod $|D|$ and $p$ respectively. Note that $\psi_{|D|}$ and $\psi_p$ are the  restrictions of $\tilde{\eta}_D$ and $\tilde{\eta}_p$ to $\mbf Z$ respectively. Now it is enough to prove that $\psi_p$ is non trivial for some restriction of $\tilde{\eta}$. We proceed as follows.

Let $n\in \mbf Z$ be such that $(n,p)=1$ and $n\not\equiv\pm 1\pmod p$. Choose $m\in\mbf Z$ such that $m\equiv n\pmod p$ and $m\equiv 1\pmod {|D|}$. Then $\psi(m)=\psi_p(m)=\psi_p(n)$. Now summing over all $\tilde{\eta}$ over $\eta$ we get
\begin{equation*}
\sum_{\tilde{\eta} \text{ over } \eta} \psi_p(m)=\sum_{\tilde{\eta} \text{ over } \eta} \tilde{\eta}_p(m)=\sum_{\tilde{\eta} \text{ over } \eta} \tilde{\eta}(m).
\end{equation*}
Now for the last sum we have $\sum_{\tilde{\eta} \text{ over } \eta} \tilde{\eta}(m)=\sum_{\xi\in \widehat{ \tilde G/G}}\tilde{\eta}_0(m)\xi(m)=\#(\tilde G/G)\tilde{\eta}_0(m)\delta_{G}(m)$, where $\delta_{G}(s)=1$ if $s\in G$, $0$ otherwise and $\tilde{\eta}_0$ is a fixed extension of $\eta$ to $\tilde{G}$.  Thus we have
\begin{equation*}
\sum_{\tilde{\eta} \text{ over } \eta} \psi_p(m)=\#(\tilde G/G)\tilde{\eta}_0(m)\delta_{G}(m).
\end{equation*}
Clearly by our choice of $m$ and $n$ we have $N(m)=m^2\not\equiv 1\pmod{|D|p}$. Thus $\delta_{G}(m)=0$. Hence not all $\psi_p$ could be trivial. This completes the proof.
\end{proof}

\begin{lem}\label{lem:twists}
Let $\eta$ be a character of $G$ and $\xi, \xi' \in \tilde{G}$ be two extensions of $\eta$. Then   $h_{\xi}$ defined as in (\ref{eq:heta}) is zero if and only if $h_{\xi'} =0$.
\end{lem}
\begin{proof}
Let us note that we can write $h_\xi(\tau) := \sum_{s \bmod i\D p, \, (s,i\D p)=1} \overline{\xi(s)} h_{s}(|D|p \tau)$ as
\[ h_\xi(\tau) = \sum_{s \in \tilde{G} / G} \overline{\xi(s)} \left( \sum_{\mu \in G}  \overline{\eta(\mu)} h_{\mu s}(|D|p \tau) \right) \]
and similarly for $h_{\xi'}$. Now $h_\xi=0$ implies that each of the terms (let us call them $f_s$) in the braces above are zero. This can be checked from the shape of the Fourier expansion of the $f_s$'s. Namely, the Fourier expansion of $f_s$ is supported on all $n$ such that $n \equiv - N(s) \bmod |D|p$ and no two norms of two distinct elements $s_1,s_2$ from $\tilde{G}/G$ with $(s_1s_2,i\D p)=1$ can be congruent modulo $|D|p$ (cf. end of proof of \propref{pnonsplit}). Since $f_s$ does not depend on $\xi$, this proves the lemma.
\end{proof}

\begin{proof}[\textbf{Proof of \propref{prop:hetanonzero}.}]
This is an immediate consequence of lemma \ref{lem:nontrivialeta} and lemma \ref{lem:twists}.
\end{proof}

\section{The case of elliptic Modular Forms} \label{int}
\subsection{Proof of Theorem \ref{th:02}(a)}
\begin{thm}
Let $\chi$ be a Dirichlet character of conductor $m_\chi$ and $N$ be a positive integer such that $N/m_\chi$ is square-free. Let $f\in S_k(N,\chi)$ be such that $a(f,n)=0$ for all but finitely many square-free integers $n$. Then $f=0$.
\end{thm}
\begin{proof}
$f\in S_k(N,\chi)$ is a newform then the result follows from multiplicity-one. Let $f\in S_k(N,\chi)$ be non-zero. Consider a basis $\{f_{1}, f_{2},......f_{s}\}$ of newforms of weight $k$ and level dividing $N$. Let their Fourier expansions be given by $f_{i}(\tau)=\sum_{n=1}^{\infty} b_{i}(n)q^{n}$. Then for all primes $p$, one has $T_{p}f_{i}=b_{i}(p)f_{i}$. By "multiplicity-one", if $i\neq j$, we can find infinitely many primes $p>N$ such that $b_{i}(p)\neq b_{j}(p)$. Now by the theory of newforms, there exist $\alpha_{i,\delta} \in \mf C$ such that $f(\tau)$ can be written uniquely in the form
\begin{equation}\label{eq:decomp}
f(\tau)=\sum_{i=1}^{s} \sum_{\delta |N} \alpha_{i,\delta}f_{i}(\delta \tau).
\end{equation}
Since $f\neq 0$, we may, after renumbering the indices, assume $\alpha_{1,\delta}\neq 0$ for some $\delta|N$. Let $p_{1}\nmid N$ be any prime for which $b_{1}(p_{1})\neq b_{2}(p_{1})$. Then consider the form $g_{1}(\tau)=\sum_{n=1}^{\infty}a_{1}(n)q^{n} :=T_{p_1}f(\tau)-b_{2}(p_{1})f(\tau)$ so that
\begin{align*}
g_{1}(\tau)=\sum_{i=1}^{s}(b_{i}(p_{1})-b_{2}(p_{1}))\sum_{\delta|N} \alpha_{i,\delta}f_{i}(\delta \tau).
\end{align*}
The cusp forms $f_{2}(\delta \tau)$ for any $\delta \mid N$, do not appear in the decomposition of $g_{1}(\tau)$ but $f_{1}(\delta \tau)$ does for some $\delta |N$. Also it is easy to see that $a_{1}(n)=a(f,p_1n)+\chi(p_1)p_1^{k-1}a(f,n/p_1)-b_{2}(p_{1})a(f,n)$. Proceeding inductively in this way, we can remove all the non-zero newform components $f_{i}(\delta \tau)$ for all $i=2,...,s$, to obtain a cusp form $F(\tau)$ in $S_{k}(N,\chi)$. After dividing by a suitable non-zero complex number we get
\begin{equation*}
F(\tau)=\sum_{n=1}^{\infty}A(n)q^{n}:=\sum_{\delta|N}\alpha_{1,\delta}f_{1}(\delta \tau).
\end{equation*}
Now by repeating the above steps we get finitely many algebraic numbers $\beta_{j}$ and positive rational numbers $\gamma_{j}$ such that for every $n$
\begin{equation} \label{f}
A(n)=\sum_{\delta|N} \alpha_{1,\delta} b_{1}(n/\delta)=\sum_{j} \beta_{j}a(f,\gamma_{j}n).
\end{equation}
Let $\delta_{1}$ be the smallest divisor of $N$ such that $\alpha_{1,\delta_{1}}\neq 0$ in (\ref{eq:decomp}) and let $F^{*}(\tau)=U_{\delta_{1}}F(\tau)$. Then $F^{*}(\tau)\in S_{k}(N,\chi)$ with $F^{*}(\tau)=\sum_{n=1}^{\infty} A(\delta_{1}n)q^{n}$.

Since $f_{1}\neq 0$ there are infinitely many primes $p$ such that $b_{1}(p)\neq 0$. Let $S=\{p:p\hspace{1 mm}$ prime, $p|N\}\cup \{p: p\hspace{1 mm}$ prime, $b_{1}(p)=0\}\cup \{\text{the primes } p_i \text{ chosen as above}\}$.

If $p\notin S$, then $A(\delta_{1}p)=\alpha_{1,\delta_{1}}b_{1}(p)\neq 0$ and there are infinitely many such primes. For each of these $p$ we get a $j=j_0$ such that $a(f,\gamma_{j_0}\delta_{1}p)\neq 0$. Let us now finish the proof of the theorem.

Let $m_{1}|N$ be such that $f_{1}(\tau)$ is a newform in $S_{k}(m_{1},\chi_{1})$, where $\chi_{1}$ (mod $m_{1}$) is the character induced by $\chi$ (mod $N$). Then $m_\chi|m_{1}$ and for each $\delta$ in the sum (\ref{eq:decomp}), $\delta m_{1}|N$. Since $N/m_\chi$ is square-free we must have that each of the $\delta$ in (\ref{eq:decomp}) is square-free (since $\delta|(N/m_\chi)$). In particular $\delta_{1}$ is square-free. Next, in the process of obtaining $F$ as above, clearly we can choose primes $p_{1},p_{2}....$ pairwise distinct (by multiplicity-one).  By construction, the prime divisors of any $\gamma_j$ appearing in \eqref{f} are from the set $\{p_1,p_2,\ldots,p_s  \}$. Moreover, since the highest power of a $p_i$ ($i=1,2,\ldots,s$) is either $0,\pm 1$, all the $\gamma_{j}$'s are square-free. In particular $\gamma_{j_0}$ is square-free and $\delta_1,p,\gamma_{j_0}$ are pairwise co-prime. The result thus follows with $n=\gamma_{j_0}\delta_{1}p$ with any $p \not \in S$.
\end{proof}

\subsection{Second moment of square--free Fourier coefficients} \label{2mom}
\thmref{th:02} would be proved by studying the second moment of the Fourier coefficients of an integral weight cusp form. We first recall the following well known result due to Rankin \cite{rankin} and Selberg \cite{selberg}.
\begin{thm}\label{th:3}
Let $f\in S_{k}(N,\chi)$ be non-zero. Then there exists a constant $A_{f}>0$ such that
\begin{equation}
\underset{n\leq X}{\sum}|a'(f,n)|^{2}=A_{f}X+O(X^{\frac{3}{5}}).
\end{equation}
Moreover, $A_{f}=\frac{3}{\pi}\frac{(4\pi)^{k}}{\Gamma(k)}\big[\mrm{SL}_2(\mathbf{Z})\colon \Gamma_0(N)\big]^{-1}\left\langle f,f\right\rangle_N$. The implied constant depends only on $f$.
\end{thm}
The following is a first step towards the proof of \thmref{th:02}, adapted from Saha\cite{saha}.
\begin{prop}\label{prop:M}
Let $N$ be a positive integer and $\chi$ be Dirichlet character $\mod N$ whose conductor is $m_\chi$. Let $f\in S_k(N,\chi)$ be non-zero and $a(f,n)=0$ whenever $(n,N)>1$. Let $M$ be a fixed square-free integer such that $M$ contains all the primes dividing $N$. Then there exists $B_{f,M}>0$ such that
\begin{equation}
\underset{(n,M)=1}{\sum_{n\le X}}|a'(f,n)|^2=B_{f,M}X+O(X^{\frac{3}{5}}).
\end{equation}
\end{prop}
\begin{proof}
Define $g(\tau)=\sum_{(n,M)=1} a(f,n)q^n$. Let $p_1$,$p_2$,...$p_t$ be the primes in $M$ that do not divide $N$ and $M_0$ be such that $M=M_0p_1p_2....p_t$. Then $g\in S_k(NM^2/M_0,\chi)$ (see \cite{miyake}).
If $g\neq 0$ then 
\begin{equation*}
\sum_{n\leq X, (n,M)=1}|a'(f,n)|^2=\sum_{n\leq X}|a'(g,n)|^2=A_gX+O(X^{\frac{3}{5}}),
\end{equation*}
where $A_g$ is as in Theorem \ref{th:3}. Put $B_{f,M}:=A_g$. Since $g\neq 0$, we have $B_{f,M}>0$.

We now prove that $g\ne 0$. Indeed, let $g_0=f$.
Let $g_1(\tau)=\sum_{(n,p_1)=1}a(g_0,n)q^n$. Then $g_1\in S_k(Np_1^2,\chi)$ (see \cite[page 157]{miyake}). If $g_1=0$, then $a(f,n)=0$ for every $(n,p_1)=1$, which in turn implies $(p_1, N/m_{\chi})>1$, which is impossible.

For each $1\le j\le t$ construct $g_j$ as $g_j(\tau)=\sum_{(n,p_j)=1}a(g_{j-1},n)q^n.$
Then $g_j\in S_k(Np_1^2....p_j^2,\chi)$. If for any $1\le j\le t$, $g_j=0$ but $g_{j-1}\neq 0$, then $a(g_{j-1},n)=0$ for $(n,p_j)=1$. This would mean $(p_j, Np_1^2....p_{j-1}^2/m_\chi)>1$, which is impossible. Hence $g_j\neq 0$ for $1\le j \le t$.

We have from the definition of $g$ that,
\begin{equation*}
g(\tau)=\underset{(n,M_0)=1}{\sum}a(g_t,n)q^n.
\end{equation*}
If $g=0$, then $a(g_t,n)=0$ whenever $(n,M_0)=1$, that is $a(f,n)=0$ whenever $(n,M_0)=1$, consequently $f=0$ which is not possible. Thus 
$g\neq 0$.
\end{proof}
\begin{cor}\label{cor:nr}
Let $f\in S_{k}(N,\chi)$ be non-zero. Then for any $r$ with $(r,N)=1$, there exists a constant $A_{f,r}>0$ depending only on $f$ and $r$  such that
\begin{equation}\label{eq:5}
\underset{n\leq X}{\sum}|a'(f,nr)|^{2}=A_{f,r}X+O(X^{\frac{3}{5}}).
\end{equation}
Moreover, $A_{f,r}=\frac{3}{\pi}\frac{(4\pi)^{k}}{\Gamma(k)}\big[\mrm{SL}_2(\mathbf{Z})\colon\Gamma_0(Nr)\big]^{-1}r^{1-k}\left\langle U_rf, U_rf\right\rangle_{Nr}$.
\end{cor}
\begin{proof}
Consider the Hecke operator $U_r$ acting on $f$, $U_rf=\underset{n>0}{\sum}a(f,nr)e(n\tau)$.
Let $g=U_rf$, then we have that $g\in S_{k}(Nr,\chi)$. Now applying Theorem \ref{th:3} to $g$ we get (\ref{eq:5}).
\end{proof}

\subsubsection{Some bounds for Peterson norms} \label{pet-bd}
In order to get an estimate for $A_{f,r}$ in (\ref{eq:5}) we slightly modify a result by J. Brown and K. Klosin \cite{brown} to include characters and use it to get an expression for $ \left\langle U_rf, U_rg\right\rangle$ . The following results might be of independent interest also. 
\begin{thm}\label{th:4}
For $p\nmid N$, let $f,g\in S_{k}(N,\chi)$ be eigenfunctions for the Hecke operator $T_{p}$ with eigenvalues $\lambda_{f}(p)$ and $\lambda_{g}(p)$ respectively. Then
\begin{equation}
\left\langle U_pf,U_pg\right\rangle_{Np} =\left(p^{k-2}+\frac{(p-1)\lambda_{f}(p)\overline{\lambda_{g}(p)}}{p+1} \right) \left\langle f,g\right\rangle_{Np}.
\end{equation}
\end{thm}
\begin{proof}
We have $U_pf=T_pf-\chi(p)p^{\tfrac{k}{2}-1}f|B_p$, where $B_{p}$ is the matrix $\begin{psm}
p & 0\\
0 & 1
\end{psm}$. Thus $\left\langle U_pf,U_pf\right\rangle_{Np}$ is given by $\left(\lambda_f(p)\overline{\lambda_{g}(p)}+p^{k-2} \right)\left\langle f,g\right\rangle_{Np}-p^{\frac{k}{2}-1}\left(\overline{\lambda_g(p)}\chi(p)\left\langle f|B_p,g\right\rangle_{Np}+\lambda_f(p)\overline{\chi(p)\left\langle g|B_p,f\right\rangle_{Np}} \right)$, where $B_{p}$ is the matrix $\begin{psm}
p & 0\\
0 & 1
\end{psm}$.

Now we evaluate $\left\langle f|B_p,g\right\rangle_{Np}$ and $\left\langle g|B_p,f\right\rangle_{Np}$. Since $\begin{psm}
1 & j\\
0 & p
\end{psm}=\begin{psm}
1 & 0\\
0 & p
\end{psm}\begin{psm}
1 & j\\
0 & 1
\end{psm}$, we get
\begin{equation*}
p^{1-\frac{k}{2}}\left\langle T_pf,g \right\rangle_{Np}= \sum_{j=0}^{p-1}\left\langle f|\begin{psm}
1 & 0\\
0 & p
\end{psm},g\right\rangle_{Np}+\chi(p)\left\langle f|B_p,g\right\rangle_{Np}.
\end{equation*}
Now there exists $a,b\in\Gamma_0(N)$ such that $a\begin{psm}1 & 0\\0 & p\end{psm}b=\begin{psm}p & 0\\0 & 1\end{psm}$ and proceeding as in \cite{brown}, we get that
\begin{equation}\label{eq:fbp}
\left\langle f|B_p,g\right\rangle_{Np}=p^{1-\frac{k}{2}}\frac{\lambda_f(p)}{\chi(p)(p+1)}\left\langle f,g \right\rangle_{Np}.
\end{equation}
We get a similar expression for $\left\langle g|B_p,f\right\rangle_{Np}$.

Putting everything together we get
\begin{equation*}
\left\langle U_pf,U_pg\right\rangle_{Np} =\left( p^{k-2}+\frac{(p-1)\lambda_{f}(p)\overline{\lambda_{g}(p)}}{p+1}  \right)  \left\langle f,g\right\rangle_{Np}. \qedhere
\end{equation*}
\end{proof}

We now use Theorem \ref{th:4} to calculate $\left\langle U_{r^2}f, U_{r^2}g\right\rangle$, for $r$ square-free. 
\begin{prop}\label{prop:p2}
Let $p\nmid N$ and $f,g\in S_{k}(N,\chi)$ be eigenfunctions for the Hecke operators $T_p$ and $T_{p^2}$ with the eigenvalues $\lambda_{f}(p),\lambda_{g}(p)$ and $\lambda_f(p^2),\lambda_{g}(p^2)$ respectively. Then 
\begin{equation}
\left\langle U_{p^{2}}f,U_{p^{2}}g\right\rangle_{Np^2}  =  \left( \lambda_f(p^2)\overline{\lambda_{g}(p^2)}+p^{k-2}\lambda_f(p)\overline{\lambda_{g}(p)}-\frac{\lambda_f(p^2)\overline{\lambda_g^2(p)}+\lambda_f^2(p)\overline{\lambda_g(p^2)}}{p+1}   \right)  \left\langle f,g\right\rangle_{Np^2} .
\end{equation}
\end{prop}
\begin{proof}
Using the definition of $T_{p^2}$ from (\ref{hecke}) we have,
\begin{align*}
T_{p^{2}}f=U_{p^{2}}f+\chi(p)p^{\frac{k}{2}-1}(U_pf)|B_p+\chi(p^2)p^{k-2}f|B_{p^2}\\
\text{i.e., }U_{p^{2}}f=T_{p^{2}}f-\chi(p)p^{\frac{k}{2}-1}(U_pf)|B_p-\chi(p^2)p^{k-2}f|B_{p^2},
\end{align*}
where for $d \geq 1$, $B_d$ is the matrix $\begin{psm}
d & 0\\
0 & 1
\end{psm}$.
Now expanding $\left\langle U_{p^{2}}f,U_{p^{2}}g\right\rangle_{Np^2}$ using the above expression for $U_{p^{2}}f$ and using \thmref{th:4} we get the proposition.
\end{proof}
\begin{cor}\label{cor:p2}
Let $f,g\in S_{k}(N,\chi)$ be eigenfunctions for all Hecke operators $T_n$ with $(n,N)=1$ and $r$ be any square-free integer with $(r,N)=1$. If $\left\langle f,g\right\rangle_N=0$, then $\left\langle U_{r^{2}}f,U_{r^{2}}g\right\rangle_{Nr^2}=0$.
\end{cor}
\begin{proof}
First we prove by induction on the number of prime factors of $r$ that $\left\langle U_{r^{2}}f,U_{r^{2}}g\right\rangle_{Nr^2}$ equals
\begin{equation}\label{eq:ur2}
\left\langle f,g\right\rangle_{Nr^2}\cdot \prod_{p|r}\left(\lambda_f(p^2)\overline{\lambda_{g}(p^2)}+p^{k-2}\lambda_f(p)\overline{\lambda_{g}(p)}-\frac{\lambda_f(p^2)\overline{\lambda_g^2(p)}+\lambda_f^2(p)\overline{\lambda_g(p^2)}}{p+1}\right).
\end{equation}
Let $r=p_1p_2.....p_m$. For $m=1$ the result in \eqref{eq:ur2} is true from proposition \ref{prop:p2}. Now we assume \eqref{eq:ur2} to hold for $m-1$.

Let $r_1=r/p_m$ and let $f_1=U_{r_1^2}f$ and $g_1=U_{r_1^2}g$. Then $f_1,g_1\in S_k(Nr_1^2,\chi)$ and $f_1,g_1$ are eigenfunctions for $T_{p_m}$ and $T_{p_m^2}$ with the eigenvalues $\lambda_f(p_m), \lambda_g(p_m)$ and $\lambda_f(p_m^2), \lambda_g(p_m^2)$ respectively (since $U_{r_1^2}$ commutes with $T_{p_m}$ and $T_{p_m^2}$). Now using proposition (\ref{prop:p2}), $\left\langle U_{p_m^2}f_1,U_{p_m^2}g_1\right\rangle_{Nr^2}$ equals
\begin{equation*}
\left\langle f_1,g_1\right\rangle_{Nr^2}\cdot \left(\lambda_f(p_m^2)\overline{\lambda_{g}(p_m^2)}+p_m^{k-2}\lambda_f(p_m)\overline{\lambda_{g}(p_m)}-\frac{\lambda_f(p_m^2)\overline{\lambda_g^2(p_m)}+\lambda_f^2(p_m)\overline{\lambda_g(p_m^2)}}{p_m+1}\right).
\end{equation*}
The proof of (\ref{eq:ur2}) follows now by induction, and \corref{cor:p2} is immediate.
\end{proof}
For any positive integer $r$, let $\omega(r)$ denote the number of distinct primes dividing $r$. We have the following corollary.
\begin{cor}\label{cor:rsq}
Let $f\in S_k(N,\chi)$ be an eigenfunction of the Hecke operators $T_n$ for all $(n,N)=1$ with the corresponding eigenvalues $\lambda_f(n)$ and $r$ be a square-free integer with $(r,N)=1$. Then
\begin{equation}\label{eq:rsq}
\left\langle U_{r^2}f, U_{r^2}f\right\rangle_{Nr^2}\le 19^{\omega(r)}r^{2k-2}  {\left\langle f,f\right\rangle_{Nr^2}} .
\end{equation}
\end{cor}
\begin{prop}\label{prop:arbf}
Let $f\in S_k(N,\chi)$ and $r$ be a square-free integer with $(r,N)=1$. Then
\begin{equation}
\left\langle U_{r^2}f, U_{r^2}f\right\rangle_{Nr^2}\le 19^{\omega(r)}r^{2k-2}\left\langle f,f\right\rangle_{Nr^2}.
\end{equation}
\end{prop}
\begin{proof}
Let $\{f_i\}_{i=1}^s$ be a orthogonal basis for $S_k(N,\chi)$ such that $f_i$ is an eigenfunction for Hecke operators $T_n$ for all $(n,N)=1$. Write $f(\tau)=\sum_i^s c_if_i(\tau)$. Then using the orthogonality property from corollary \ref{cor:p2}, $\left\langle U_{r^2}f, U_{r^2}f\right\rangle_{Nr^2}=\tsum |c_i|^2\left\langle U_{r^2}f_i,U_{r^2}f_i \right\rangle_{Nr^2}$.

Using corollary \ref{cor:rsq} we get
\begin{equation*}
\left\langle U_{r^2}f, U_{r^2}f\right\rangle_{Nr^2}\le 19^{\omega(r)}r^{2k-2}\tsum_i |c_i|^2\left\langle f_i,f_i \right\rangle_{Nr^2}=19^{\omega(r)}r^{2k-2}\left\langle f,f\right\rangle_{Nr^2}.\qedhere
\end{equation*}
\end{proof}
\noindent An immediate consequence is the following.
\begin{cor}\label{cor:Af}
Let $A_{f,r}$ be as in (\ref{eq:5}) and $r=s^2$ where $s$ is a square-free integer with $(s,N)=1$. Then $A_{f,r}\le 19^{\omega(s)}A_f$.
\end{cor}
\begin{proof}
From the expression for $A_{f,r}$ in corollary \ref{cor:nr} and proposition \ref{prop:arbf} we get \[A_{f,r}\le \frac{3}{\pi}\frac{(4\pi)^{k}}{\Gamma(k)}\big[\mrm{SL}_2(\mathbf{Z})\colon \Gamma_0(Nr)\big]^{-1}19^{\omega(s)}\left\langle f,f\right\rangle_{Nr}.\] Since $\left\langle f,f\right\rangle_{Nr}= r\underset{p|r}{\tprod}(1+\frac{1}{p})\left\langle f,f\right\rangle_N$, we get the required bound.
\end{proof}
Let $\mathcal{S}$ denote the set square-free positive integers and let $\mathcal{S}_M\subset\mc S$ denote those which are coprime to an integer $M$. We now proceed as in \cite{saha} to prove \thmref{th:02}.
\begin{prop}\label{prop:sfmx}
Let $N$ be a positive integer and $\chi$ be Dirichlet character mod $N$. Let $f\in S_k(N,\chi)$ be non-zero and $a(f,n)=0$ whenever $(n,N)>1$. Then there are infinitely many odd and square-free integers $n$ such that $a(f,n)\neq 0$. Moreover, for any $\varepsilon>0$, $\#\{0<n<X:n\in\mathcal{S}, a(f,n)\neq 0\}\gg_{f,\varepsilon} X^{1-\varepsilon}$.
\end{prop}
\begin{proof}
For any square-free positive integer $M$ as in proposition \ref{prop:M}, define
\begin{equation}
S_f(M,X)=\sum_{n\in \mathcal{S}_M,n\leq X}|a'(f,n)|^2.
\end{equation}
Letting $g(\tau)=\sum_{(n,M)=1} a(f,n)q^n$ and use the sieving identity 
\begin{equation*}
\sum_{r^2|n}\mu(r)=\begin{cases*}
1 \q\text{ if } n \text{ is square-free};\\
0 \q\text{ otherwise}.
\end{cases*}
\end{equation*}
for sieving the square-free terms from the Fourier expansion of $g$ to get
\begin{equation*}
S_f(M,X)=\sum_{r\in \mathcal{S}_M,r\leq X}\mu(r)\underset{(m,M)=1}{\sum_{m\leq X/r^2}}|a'(g,mr^2)|^2.
\end{equation*}
Thus for large $X$, we get from proposition \ref{prop:M}
\begin{align*}
S_f(M,X)&\ge \frac{B_{f,M}}{2}X-\sum_{r\in\mathcal{S}_M,\text{ } 2\le r\leq \sqrt{X}}\text{ }\sum_{m\leq X/r^2}|a'(g,mr^{2})|^{2}\\
&\ge \frac{B_{f,M}}{2}X-\sum_{r\in\mathcal{S}_M,\text{ }2\le r\leq \sqrt{X}}2A_{g,r^2}\frac{X}{r^2},
\end{align*}
where $A_{g,r^2}$ is as in corollary \ref{cor:nr}. Using the bound for $A_{g,r^2}$ from corollary \ref{cor:Af} and the definition of $B_{f,M}$ from proposition \ref{prop:M}, we get $A_{g,r^2}\leq 19^{\omega(r)}A_g=19^{\omega(r)}B_{f,M}$.
\begin{equation*}
S_f(M,X)\geq\Big(\tfrac{1}{2}-2\tsum_{r\ge 2,r\in\mathcal{S}_M}19^{\omega(r)}r^{-2}\Big)B_{f,M}X=\Big(\tfrac{5}{2}-2\tprod_{p\nmid M}\big(1+\tfrac{19}{p^2}\big)\Big)B_{f,M}X.
\end{equation*}
Let us choose $M$ to be the product of primes $p < 87$ and the primes dividing $N$ such that $M$ is square-free. Note that $\prod_{p>Y}\big(1+\tfrac{19}{p^2}\big)$ is bounded by $e^{\sum_{p>Y}\frac{19}{p^2}}$ which in turn bounded above by $e^{\frac{19}{Y}}$. In our case $Y\ge 87$ and so $e^{\frac{19}{Y}}<5/4$. Therefore $S_f(M,X)> B_fX$, for some $B_f>0$. Now using (\ref{eq:deligne}) it is immediate that for any $\varepsilon>0$, \[\{0<n<X:n\in\mathcal{S}, a(f,n)\neq 0\}\gg_{f,\varepsilon} X^{1-\varepsilon}.\qedhere\]
\end{proof}
\begin{rmk}
The introduction of the parameter $M$ in the previous proposition is done so as to make the quantity $\underset{r>2}{\sum_{r\in\mathcal{S}}}19^{\omega(r)}r^{-2}$ less than $1/4$.
\end{rmk}
Now we prove the result in general case by reducing it to the situation of proposition \ref{prop:sfmx}.
\subsection{Proof of Theorem \ref{th:02}} \label{partb}
Let $p_1,p_2,.....p_t$ be the distinct prime factors of $N$. We construct a sequence $\{f_i:1\le i\le t\}$ with the following properties
\begin{enumerate}
\item[(a)] $f_i\neq 0$.
\item[(b)] $f_i\in S_k(NN_i,\chi)$, where $N_i$ is composed of primes $p_1,p_2,...p_i$.
\item[(c)] $a(f_i,n)=0$, whenever $(n,p_1p_2...p_i)>1$.
\item[(d)] If there exist infinitely many square-free integers $n$ such that $a(f_i,n)\neq 0$, then same is true for $f_{i-1}$.
\item[(e)] If $\{0<n<X: n \text{ square-free, } a(f_i,n)\neq 0\}\gg X^\epsilon$, for some $\epsilon>0$, then \[\{0<n<X: n \text{ square-free, } a(f_{i-1},n)\neq 0\}\gg X^\epsilon.\]
\end{enumerate}
Let $f_0=f$. Now we construct $f_1$. If 
\begin{equation}\label{eq:f1}
\underset{(n,p_1)=1}{\sum} a(f_0,n)q^n\neq 0,
\end{equation}
then we take $f_1(\tau)=\sum_{(n,p_1)=1}a(f_0,n)q^n$. Then $f_1\in S_k(Np_1,\chi)$ (see \cite{miyake}) and satisfies all the required properties. If (\ref{eq:f1}) is not true, then $a(f_0,n)=0$ for all $(n,p_1)=1$, that is $(p_1,N/m_\chi)>1$ and $f_0(\tau)=f_{p_1}(p_1\tau)$ for some $f_{p_1}\in S_k(N/p_1,\chi)$. Since $f_0\neq 0$, we see that $f_{p_1}\neq 0$ and let
\begin{equation*}
f_1(\tau)=\underset{(n,p_1)=1}{\sum}a(f_{p_1},n)q^n.
\end{equation*}
We have $f_1\in S_k(N,\chi)$. If $f_1=0$, then $(p_1,N/(p_1m_\chi))>1$, which is impossible since $N/m_\chi$ is square-free. Thus again $f_1\neq 0$ and satisfies all the listed properties. Now we construct $f_i$ from $f_{i-1}$ inductively for $1\le i\le t$ as above. Let $f_{i-1}\in S_k(NN_{i-1},\chi)$. If
\begin{equation}\label{eq:fi}
\underset{(n,p_i)=1}{\sum}a(f_{i-1},n)q^n\neq 0,
\end{equation}
then we take 
\begin{equation*}
f_i(\tau)=\sum_{(n,p_i)=1}a(f_{i-1},n)q^n
\end{equation*}
and it satisfies all the required properties. If (\ref{eq:fi}) is not true then $(p_i,N_i/m_\chi)>1$ and $f_{i-1}(\tau)=f_{p_i}(p_i\tau)$ for some $f_{p_i}\in S_k(N_{i-1}/p_i,\chi)$. Since $f_{i-1}\neq 0$, $f_{p_i}\neq 0$ and we take
\begin{equation*}
f_i(\tau)=\underset{(n,p_i)=1}{\sum}a(f_{p_i},n)q^n.
\end{equation*}
As above $f_i\neq 0$ and satisfies all the properties. Thus we have constructed the sequence $\{f_i:1\le i\le t\}$ as claimed.
Now take $g=f_t$ and $N'=NN_t$. Then $g\in S_k(N',\chi)$ and $a(g,n)=0$ whenever $(n,N')>1$. Now we can apply proposition \ref{prop:sfmx} to $g$ and get that $a(g,n)\neq 0$ for infinitely many odd and square-free integers $n$.
The properties of the sequence $\{f_i:1\le i\le t\}$ allow us to reach $f$ from $g$ and we find that the result is true for $f$. \qed

\thmref{th:02} can be generalized to any $N$ and $\chi$, however the statement of the theorem becomes much more complicated. Here we give one sample of this when $N$ has three distinct prime factors. The following proposition indeed is necessary for us, it is used to arrive at the statement of \thmref{th:01}.

\begin{prop}\label{prop:genth2}
Let $N=p_1^{\alpha_1} p_2^{\alpha_2} p_3^2$ ($\alpha_2\ge 2$ and $0\le\alpha_1\le 1$), where $p_1,p_2,p_3$ are distinct primes and $\chi$ be a Dirichlet character mod $N$ of conductor $m_\chi$ such that $N/m_\chi=p_1^{\alpha_1}p_2^\beta p_3$ with $0\le\beta\le\alpha_2$. Let $f\in S_k(N,\chi)$ be non zero. Then there exist infinitely many odd and square-free integers $n$ with $(n,p_2)=1$ such that $a(f,p_2^\gamma n)\neq 0$, where $\gamma\le\alpha_2$ if $\alpha_2=\beta$ and $\gamma\le\alpha_2-\beta$ if $\alpha_2>\beta$. Moreover, for any $\epsilon>0$,  $\#\{0<n<X: n\in\mathcal{S}, a(f,p_2^\gamma n)\neq 0\}\gg_{f,\varepsilon} X^{1-\varepsilon}$.
\end{prop}
\begin{proof}
We construct a new cusp form $g$ such that $g$ satisfies the hypothesis of proposition \ref{prop:sfmx} or \thmref{th:02} and get the result for $f$ from the corresponding result for $g$. 

Let $\delta=\alpha_2-\beta-1$, when $\beta<\alpha_2$ and $\alpha_2-1$, when $\alpha_2=\beta$. For $0\le i\le \delta$, define $f_0=f$ and $f_i(\tau):=\sum_{n\ge 1}a(f_{i-1}, p_2n)q^n$. Since $\alpha_2\ge 2$ and $i\le \delta$, $f_i\in S_k(p_1^{\alpha_2-i}p_3^2,\chi)$. If $f_i\neq 0$ for all $1\le i\le\delta$ take $g=f_\delta$. Then $g\in S_k(p_1^{\alpha_1}p_2^{\alpha_2-\delta}p_3^2,\chi)$ and $a(g,n)=a(f,p_2^\delta n)$. Now by the definition of $\delta$, we get that the ratio $N/m_\chi =p_1^{\alpha_1}p_2p_3$, i.e., square-free. Using \thmref{th:02} we get the result for $g$ and hence for $f$. 

If $f_i=0$ for some $1\le i\le\delta$. Let $0\le i_0< \delta$ be the smallest $i$ such that $f_{i_0+1}= 0$. Then $a(f_{i_0},p_2n)=0$ for every $n\ge 1$. Thus $f_{i_0}(\tau)=\sum_{(n,p_2)=1}a(f_{i_0},n)q^n$ and we already have $f_{i_0}\in S_k(p_1^{\alpha_1}p_2^{\alpha_2-i_0}p_3^2,\chi)$. 

If $\sum_{(n,p_3)=1}a(f_{i_0},n)q^n\neq 0$, we set $g_1(\tau):=\sum_{(n,p_3)=1}a(f_{i_0},n)q^n$. Then $g_1\in S_k(p_1^{\alpha_1}p_2^{\alpha_2-i_0}p_3^3,\chi)$. If the above sum is zero, then $f_{i_0}(\tau)=\tilde{g}_1(p_3\tau)$, for some non-zero $\tilde{g}_1\in S_k(p_1^{\alpha_1}p_2^{\alpha_2-i_0}p_3,\chi)$. We set $g_2(\tau):=\sum_{(n,p_3)=1}a(\tilde{g}_1,n)q^n$. Clearly $g_2\in S_k(p_1^{\alpha_1}p_2^{\alpha_2-i_0}p_3^2,\chi)$. If $g_2=0$, then $(p_3,p_1^{\alpha_1}p_2^{\alpha_2-i_0})>1$ (see \cite{miyake}), which is impossible. Hence $g_2\neq0$.

Now let $g_1\neq 0$ (resp. $g_2\neq 0$). We repeat the above procedure with $g_1$ (resp. $g_2$) and prime $p_1$ to get $g$ such that $g$ satisfies the hypothesis of proposition \ref{prop:sfmx}.

The proof now follows by noting that $i_0\le \delta$ and that $a(g,n)=a(f,p_2^{i_0}n)$, when $(n,p_2)=1$ and $0$ otherwise.
\end{proof}

\end{document}